\documentclass{article}

\usepackage{amsmath}
\usepackage{graphicx}
\usepackage[colorlinks=true, allcolors=blue]{hyperref}
\usepackage[dvipsnames,svgnames,table]{xcolor}
\usepackage{sectsty}
\usepackage{multirow}
\usepackage{authblk}
\usepackage{relsize}
\usepackage[T1]{fontenc} 
\usepackage[english]{babel} 
\usepackage{amsmath,amsfonts,amsthm, amssymb} 
\usepackage{graphicx}
\usepackage{booktabs}
\usepackage[font=small]{subcaption}
\usepackage[font=small]{caption}
\usepackage{svg}
\usepackage{wrapfig}
\usepackage{bm}
\usepackage{soul} 
\usepackage{lipsum} 
\usepackage{sectsty} 
\usepackage{fancyhdr} 
\usepackage[T1]{fontenc} 
\usepackage[english]{babel} 
\usepackage{framed}
\usepackage{bm}
\usepackage{fancyhdr}
\usepackage{etoolbox}
\usepackage{titlesec}
\usepackage{titletoc}
\usepackage{float}
\usepackage{xcolor}

\newtheorem{theorem}{Theorem}[section]
\newtheorem{lemma}{Lemma}[section]
\newtheorem{proposition}{Proposition}[section]
\newtheorem{corollary}{Corollary}[section]

\theoremstyle{definition}

\theoremstyle{remark}
\newtheorem{remark}{Remark}

\title{Local reconstruction  of coefficients  in quantitative photoacoustic tomography}

\author{Gang Bao\thanks{Department of Mathematics, Zhejiang University, Hangzhou 310027, China (Baog@zju.edu.cn). \url{http://www.mathweb.zju.edu.cn:8080/bao/}}   \and
Mirza Karamehmedovi\'c\thanks{Department of Applied Mathematics, Technical University of Denmark, Kgs. Lyngby, Denmark (mika@dtu.dk). \url{https://www.dtu.dk/english/person/mirza-karamehmedovic?id=20347&entity=publicationsOrcid}} \and Faouzi Triki\thanks{Laboratoire Jean Kuntzmann,
Grenoble Alpes University, 700 Avenue Centrale, 38401 Saint-Martin-d'Hères, France, (Faouzi.Triki@univ-grenoble-alpes.fr). \url{https://membres-ljk.imag.fr/Faouzi.Triki/}} }

\begin{document}
\maketitle

\begin{abstract}
We study the problem of reconstructing the scattering and absorption coefficients in the radiative transport  equation using internal data in quantitative photoacoustic tomography (QPAT). In practical  settings, however, this internal data is only partially available near the boundary, owing to medium’s strong absorption and limitations of the measurement equipment. Our main contribution 
is the development of a method to recover these coefficients within a subregion where the internal 
data can be obtained with sufficient reliability.

\end{abstract} 

\section{Introduction}


Photoacoustic tomography (PAT) is a hybrid medical imaging technique which combines the high contrast of optical parameters 
with the high resolution of ultrasonic waves.
In PAT, near infra-red  photons are sent into the biological tissue which is heated up due to the 
absorption of the energy.
The heating then results in the expansion of the tissue which generates a pressure field.
The measurement of the pressure field on the boundary is
 then used to reconstruct the optical properties of the tissue.\\

\noindent The inverse problem of PAT can be decomposed into two steps \cite{beard2011biomedical,wang2017photoacoustic}.
The first step is to reconstruct the absorbed radiation map  from the measurement of ultrasonic waves on the boundary \cite{kian2024recovery,tarvainen2024quantitative,ammari2011time,kuchment2008mathematics}. 
The second step is to reconstruct the scattering  coefficient  and the absorption coefficient 
through the internal data obtained in the first step
\cite{bal2011multi,bal2012multi,ammari2011mathematical,bal2010inverse,saratoon2013gradient, naetar2014quantitative, arridge1999optical,gao2011quantitative, haltmeier2015single}. Let $\Omega \subset  {\mathbb R}^n$ ($n=2, 3$) be the domain of interest with smooth boundary $\partial\Omega$ and outward normal vector $\nu$, and let ${\mathbb S}^{n-1}$ be the unit sphere in ${\mathbb R}^d$. The propagation of near infra-red  photons in the direction $d \in
\mathbb S^{n-1}$ in biological tissues, 
with scattering and absorption coefficients $\mu_s$ and  $\mu_a$, respectively, is modeled by the radiative transport equation \cite{arridge2009optical,bal2008stability}:
\begin{equation}\label{RDT:Equation}
	\begin{array}{rcll}
  	d \cdot\nabla u(x,d) + 
  	(\mu_a(x)+\mu_s(x))u(x,d)
  	&=&  K(u)(x, d)
  	&\mbox{ in }\ \Omega\times \mathbb S^{n-1}\\
       u(x,d) &=& g(x,d) &\mbox{ on }\ \Gamma_{-},
	\end{array}
\end{equation}
where $\Gamma_{-} = \cup_{d\in \mathbb S^{n-1}}\Gamma_-(d) \times\{d\}$ and $\Gamma_{-}(d)$ ($\Gamma_+(d)$) is the incoming (outgoing) boundary defined by 
\begin{equation} \label{Def:Gamma}
\Gamma_{\pm}(d) = \left\{x\in \partial \Omega: \quad \nu(x)\cdot d\,\, \raisebox{-1ex}{$\overset{\displaystyle >}{<}$}\,\,0 \right\},
\end{equation}
and $g$ is the incoming illumination  source. The scattering operator $K$ is defined as
\begin{equation}
K(u)(x, d)= \mu_s(x)\int_{\mathbb S^{n-1}} k(d, d^\prime)u(x,d^\prime) ds(d^\prime), \quad (x,d) \in \Omega\times \mathbb S^{n-1}.
\end{equation}
where  $k(d, d^\prime)$ is non-negative, symmetric,  and represents the scattering phase function, a probability density function which describes 
the probability  that a photon traveling in a direction $d^{\prime}$ will be scattered into a direction $d$. It satisfies 
\begin{equation} \label{Cond:k}
\int_{\mathbb S^{n-1}} k(d, d^\prime)ds(d^\prime) = 1, \quad \forall d\in \mathbb S^{n-1}. 
\end{equation}
In practical applications in biomedical optics, $k$ is often taken to be the Henyey-Greenstein phase function which depends only on the product 
$d\cdot d^\prime$ \cite{arridge1999optical,welch2011optical}: 
\begin{eqnarray}
k(d, d^\prime)=  
\left\{ \begin{array}{llcc}
\frac{1}{2\pi} \frac{1-\eta^2}{1+\eta^2-2\eta d\cdot d^\prime}, \quad n=2\\
\frac{1}{4\pi} \frac{1-\eta^2}{(1+\eta^2-2\eta d\cdot d^\prime)^{\frac{3}{2}}}, \quad n=3,
\end{array}
\right.
\end{eqnarray}
\noindent
\noindent with $\eta \in (-1,1)$ denoting the dimensionless scattering anisotropy parameter. This model simplifies the 
description of light scattering by characterizing it with a single parameter $\eta$. In biological tissues, scattering is typically strongly 
forward-directed, so that $\eta \in (0,1)$ \cite{jacques2013optical}. In other words, light propagating through tissue tends to preserve its 
original direction of propagation. As $\eta$ approaches $1$, scattering becomes increasingly forward-peaked, leading to a narrower and more collimated beam.
As will be shown in Section \ref{Sec:forward}, the exponential decay rate of the illumination intensity along the beam propagation direction depends on 
how close $\eta$ is to $1$. The anisotropy parameter $\eta$ is also used to define the so-called effective scattering coefficient through $\mu_s^\prime$ = $(1-\eta) \mu_s$. Notice that  $\mu_s^\prime$ tends to zero when $\eta$ 
approaches $1$. 

\begin{figure}
    \centering
    \includegraphics[width=0.49\linewidth]{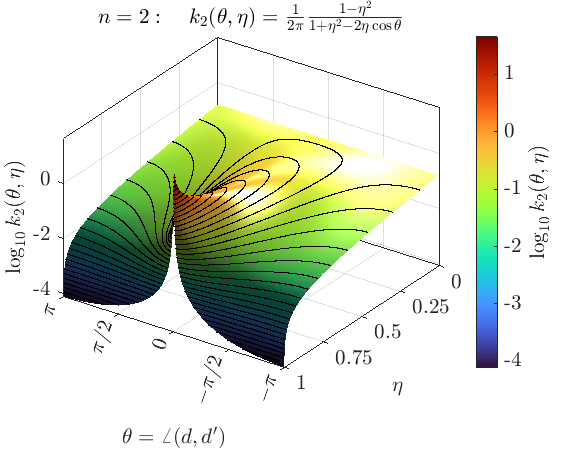}        \includegraphics[width=0.49\linewidth]{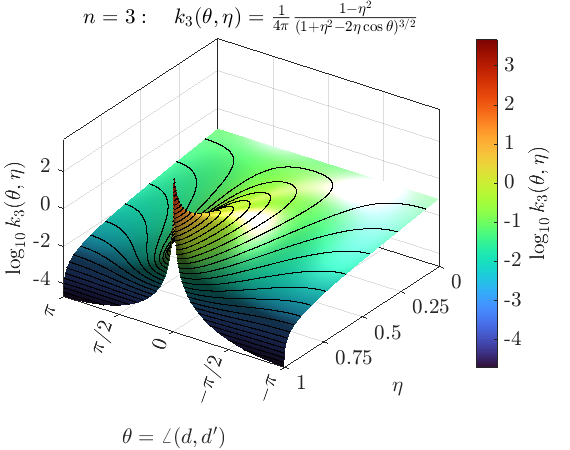}
    \caption{The Henyey-Greenstein phase function describes increasingly forward-peaked scattering as $\eta\nearrow1$.}
    \label{fig:placeholder}
\end{figure}

\noindent The absorbed optical energy density $H(x,g)$ for a given boundary illumination $g$, is defined by
\begin{equation} \label{Data}
H(x, g) = \mu_a(x)\int_{\mathbb S^{n-1}} u(x, d)ds(d), \quad x\in \Omega.
\end{equation}

\noindent The optical inverse problem of QPAT is to recover $(\mu_a, \mu_s)$ from the knowledge of the map $g\mapsto H(\cdot, g)$. \\

\noindent In reality,  the recovered internal data takes the form  $\Upsilon(\cdot) H(\cdot ,g)$,  where $\Upsilon$ is the non-dimensional Gr\"uneisen coefficient, representing the photo-acoustic efficiency of the tissue in the present model. In \cite{bal2011multi,bal2012multi}, the authors showed, within the framework of a diffusion model, that the Gr\"uneisen parameter $\Upsilon$ and the optical coefficients $(\mu_a,\mu_s)$ cannot be simultaneously recovered unless additional information is available. In this work, we assume that $\Upsilon$ is known and focus exclusively on the reconstruction of the optical coefficients $(\mu_a,\mu_s)$.\\

\noindent In practice, numerous experimental studies have demonstrated that the imaging depth of photoacoustic imaging -defined as the maximum depth at which structures can be resolved with a prescribed spatial resolution- remains relatively limited, typically to only a few centimeters \cite{liu2024imaging}. This limitation primarily stems from the restricted penetration of diffusive near-infrared photons, whose intensity is strongly attenuated by both absorption and scattering mechanisms, a difficulty that is also encountered in optical tomography \cite{arridge1999optical}. As a consequence, the generated ultrasound signals decay rapidly with increasing depth, resulting in a significant deterioration of image quality in deeper regions. This phenomenon has been rigorously established from a mathematical perspective in \cite{bonnetier2022stability,triki2021holder}.\\
\noindent Most existing reconstruction methods rely on the solution of nonlocal partial differential equations and therefore require measurements throughout the entire imaging domain, including regions where the acoustic signals are weak and heavily contaminated by noise \cite{ren2006frequency,haltmeier2015single,bal2011multi,bal2012multi}. Consequently, the reconstructed images may suffer from poor overall quality, even when the structures of interest lie within shallower regions where the desired spatial resolution can still be achieved.  \\

\noindent The objective of this work is to simplify the reconstruction of the optical coefficients in the region of interest. We begin by introducing the function spaces on which the map ($g \mapsto H(\cdot,g)$) is rigorously defined in Section \ref{Sec:forward}. We then establish the well-posedness of the forward problem. More precisely, using its integral formulation, we derive in Theorem \ref{main1} a new weighted $L^\infty$-estimate for the solution of the radiative transport equation \eqref{RDT:Equation}. This estimate, in particular, demonstrates the exponential decay of the solution along characteristics.
In Section \ref{Sec:asymptotic}, we derive asymptotic expansions for both $u$ and $H$ with respect to the small depth parameter $\tau_-(x,d)$, which measures the distance to the inflow boundary along characteristics (Theorem \ref{tD:asymptotic} and Corollary \ref{tH:asymptotic}). The leading-order term of the measured data identifies the local absorption coefficient, while the next-order correction captures the contribution of scattering. This analysis yields an explicit local approximation of the measurement map in terms of $\mu_a$, $\mu_s$, and source-dependent directional moments. These leading-order expansions are then exploited in Section \ref{Sec:inverse}   to solve the inverse problem. Finally, numerical experiments presented in Section \ref{sec:numerical} validate the asymptotic model and demonstrate stable recovery of the optical coefficients in shallow subregions near the boundary.

\section{The Forward Problem} \label{Sec:forward}
Here we will study the well-posedness of the system \eqref{RDT:Equation}. For general coefficients, the forward problem  may not be uniquely solvable.
We further assume that 
\begin{itemize}
\item[i)] $g\in L^\infty(\Gamma_-) $  where $\Gamma_- = \cup_{d\in \mathbb S^{n-1}}\Gamma_-(d) \times\{d\}$. 
\item[ii)] $ \mu_a, \mu_s \in L^\infty(\Omega)$.
\item[iii)] $\mu_s\geq 0$ and $\mu_a\geq c_0$ for some positive constant $c_0$. 
\end{itemize}

Let us introduce some notations. For $d\in \mathbb S^{n-1}$, set $$\tau_\pm(x,d) = \min\left\{t\geq 0: x\pm td \notin \Omega\right\},$$ and write $\tau(x,d) = \tau_-(x,d)+\tau_+(x,d)$ for the length of the line segment through $x$ in the direction $d$ completely contained in $\Omega$. Precisely, $\tau_+(x,d)$  (resp. $\tau_-(x,d)$) is the distance it takes a particle at  $x$ traveling in direction $d$ (resp. $-d$) to reach the boundary of the domain. \\

\noindent 
We further assume that for every point $x\in \Omega$ and any direction $d\in \mathbb S^{n-1}$,  there exists $(x^\prime, d) \in \Gamma_-$ such that 
$$
x= x^\prime +td, \quad   \textrm{with} \quad  t\in (0, \tau(x,d)).
$$

\noindent Notice that
\begin{eqnarray*}
 \tau_\mp(x^\prime \pm td,d)=t,  \qquad \forall t\in (0, \tau_\pm(x^\prime, d)),\; \forall x^\prime \in \Gamma_\mp(d),\\
 \tau(x^\prime \pm td,d) = \tau(x^\prime,d), \qquad \forall t\in (0, \tau_\pm(x^\prime, d)),\; \forall x^\prime \in \Gamma_\mp(d).
\end{eqnarray*}

\noindent Denote
\[
\mu = \mu_a+\mu_s,
\]
the coefficient of the zero order in the radiative transport equation. To proceed with the analysis, we reformulate the radiative transport equation \eqref{RDT:Equation}  as an equivalent integral equation. For all $(x', d)\in\Gamma_-$, and $t\in[0,\tau(x',d)]$, consider an inward extension of the boundary datum
\begin{eqnarray}\label{eq:extension}
(J g)(x^\prime+td, d) = e^{-\int_0^t \mu(x^\prime+s d) ds} g(x^\prime,d), 
\end{eqnarray}
as well as a lifting
\begin{align}\label{eq:explicitA}
Lf(x^\prime +td,d) = \int_0^t e^{-\int_s^t \mu(x^\prime+rd) dr } f(x^\prime+s d,d) ds. 
\end{align}
By elementary calculations one can verify that 
\begin{equation}
(d \cdot \nabla + \mu) J g \;= 0, \qquad  Jg |_{\Gamma_-} = g,
\end{equation}
and 
\begin{equation}
(d \cdot \nabla  +\mu) L f \;= f, \qquad  L f |_{\Gamma_-} = 0.
\end{equation}
This means that the extension $J g$ of the boundary datum is in the kernel of the differential operator and that the 
lifting $L$ is its right inverse.
The radiative transfer problem is thus equivalent with the integral equation \cite{bal2008stability,choulli1999inverse,egger2014stationary}
\begin{align} \label{eq:contraction}
u = L K u + J g,
\end{align}
and the unique solvability of \eqref{RDT:Equation} is equivalent with the invertibility of the linear operator $I- LK$ in a given Banach space. To show the latter, we use the following  technical lemmas.\\
For $ d\in \mathbb S^{n-1}$ and $ x\in \Omega$, denote $\Delta_-(x,d) = \{x-td: \; t\in [0, \tau_-(x,d)] \}$ the segment joining $x$
to $x^\prime= x-\tau_-(x,d)d$ on $\Gamma_-(d)$.
\begin{lemma}[]\label{lem:contraction1}
Let $\alpha \in L^\infty(\Omega\times \mathbb S^{n-1})$ satisfy
\begin{equation*}
0\leq \alpha(x, d) \leq \min_{\Delta_-(x,d)} \mu_a, \quad \forall x \in \Omega, \; \forall d \in \mathbb S^{n-1}.
\end{equation*}

  For any $\phi\in L^\infty(\Omega\times \mathbb S^{n-1})$  there holds
  \begin{equation} \label{contraction1}
    \|e^{\alpha \tau_-} L (\mu_s\phi)\|_{L^\infty(\Omega \times {\mathbb S}^{n-1} )}\leq \big(1-e^{-\|\tau_- \mu_s \|_{L^\infty}}\big) \|e^{\alpha \tau_-} \phi\|_{L^\infty(\Omega \times {\mathbb S}^{n-1} )}.
  \end{equation}
\end{lemma}
\begin{proof}
Let $x\in \Omega$, $d\in \mathbb S^{n-1}$, and  set $x^\prime= x- \tau_-(x,d)d$.
 For $0 < t < \tau_-(x,d)$, we have 
\begin{eqnarray*} 
 |(L(\mu_s\phi)(x^\prime + td,d)| =  \left|\int_0^t e^{-\int_s^t \mu(x^\prime+rd)dr}\mu_s(x^\prime+sd)  \phi(x^\prime+sd)ds\right|\\
    \leq \int_0^t e^{-\alpha(t-s)}  e^{-\int_s^t \mu_s(x^\prime+rd)dr} \mu_s(x^\prime+sd) \left| \phi(x^\prime+sd)\right| ds.
\end{eqnarray*}
 Therefore, for each $t\in[0,\tau_-(x,d)]$ we have
  \begin{eqnarray*}
    |e^{\alpha t} (L (\mu_s\phi))(x^\prime + td,d)| \leq 
    \int_0^t e^{-\int_s^t \mu_s(x^\prime+rd)dr}\mu_s(x^\prime+sd) e^{\alpha s}
    \left|\phi(x^\prime+sd)\right| ds\\
    \leq
    \int_0^t e^{-\int_s^t \mu_s(x^\prime+rd)dr}\mu_s(x^\prime+sd)ds \|e^{\alpha \tau_-} \phi\|_{L^{\infty}(\Omega \times {\mathbb S}^{n-1} )}.
          \end{eqnarray*}
  We finally obtain 
        \begin{eqnarray*}
 |e^{\alpha \tau_-(x,d)}(L(\mu_s\phi)(x,d)| \leq \big(1-e^{-\|\tau_- \mu_s \|_{L^\infty}} \big)  \|e^{\alpha \tau_-} \phi\|_{L^{\infty}(\Omega \times {\mathbb S}^{n-1} )}.
  \end{eqnarray*}
  
\end{proof}

\begin{remark}
The estimate \eqref{contraction1} is optimal since it coincides with the trivial explicit bound when $\mu_a$ is a constant
function. In   \cite{egger2014stationary}, the case $\alpha=0$ is considered. 
\end{remark}



\begin{lemma}\label{lem:contraction_infty}
Let $\alpha \in \mathbb R_+$ satisfy
\begin{equation*}
0\leq \alpha \leq \min_{\Omega} \mu_a. 
\end{equation*}

  For any $\phi\in L^\infty(\Omega\times \mathbb S^{n-1})$ there holds
  \begin{equation} \label{contraction}
    \| e^{\alpha \tau_-} L K \phi\|_{L^\infty(\Omega \times {\mathbb S}^{n-1} )} \leq \big(1-e^{-\|\tau_- \mu_s \|_{L^\infty}}\big) \kappa(\alpha,\eta)\|e^{\alpha \tau_-} \phi\|_{L^\infty(\Omega \times {\mathbb S}^{n-1} )},
  \end{equation}
  where $\kappa$ is defined by
  \begin{equation} \label{eq:kappa}
 \kappa(\alpha, \eta) \;=\; \max_{d\in \mathbb S^{n-1}, x\in \Omega} \int_{\mathbb S^{n-1}} k(d,d^\prime) e^{\alpha (\tau_-(x, d)- \tau_-(x, d^\prime))}ds(d^\prime).
  \end{equation}

\end{lemma}
\begin{proof}
  Using $f= K \phi$ in \eqref{eq:explicitA} and  following the same steps in the proof of Lemma \ref{lem:contraction1}, we obtain for $0 < t < \tau_-(x,d)$ 
  \begin{eqnarray*}
    |e^{\alpha t}(L K \phi)(x^\prime + td,d)| \\
    \leq \int_0^t e^{-\int_s^t \mu_s(x^\prime+rd)dr}\mu_s(x^\prime+sd) \int_{\mathbb S^{n-1}} k(d,d^\prime) 
    e^{\alpha s} |\phi(x'+sd, d^\prime)| ds(d^\prime) ds\\
    \leq\int_0^t e^{-\int_s^t \mu_s(x^\prime+rd)dr}\mu_s(x^\prime+sd) ds\, \kappa(\alpha,\eta) \|e^{\alpha \tau_-}\phi\|_{L^{\infty}(\Omega \times {\mathbb S}^{n-1} )} \\
    \leq \big(1-e^{-\|\tau_- \mu_s \|_{L^\infty}} \big)  \kappa(\alpha,\eta) \|e^{\alpha \tau_-}
    \phi\|_{L^{\infty}(\Omega \times {\mathbb S}^{n-1} )},
  \end{eqnarray*}
  which finishes the proof.
  \end{proof}

\begin{proposition}[\cite{egger2014stationary}] 
      Under the above assumptions, the forward problem \eqref{RDT:Equation} has a unique solution $u$ that satisfies
  \begin{equation} \label{forward}
    \| u \|_{L^\infty(\Omega\times{\mathbb S}^{n-1})} \leq e^{C} \|g\|_{L^\infty(\Gamma_-)},
  \end{equation}
where $C = \| \tau_- \mu_s\|_{L^\infty(\Omega\times {\mathbb S}^{n-1})}$.
  
\end{proposition}
\proof 

Since $\kappa(0, \eta) = 1$, we deduce from Lemma~\ref{lem:contraction_infty}  for $\alpha =0$ that the operator $LK$ is a contraction of the unity on  $L^\infty(\Omega \times {\mathbb S}^{d-1} )$, and hence the  integral equation \eqref{eq:contraction} has a unique solution 
$u \in L^\infty(\Omega\times \mathbb S^{n-1})$ whenever $J g$ lies in $L^\infty(\Omega \times {\mathbb S}^{n-1} )$. \\

Since $\mu \geq 0$, we immediately obtain $|J g (x^\prime+t d, d)| \le |g(x^\prime,d)|$, which yields
\[
\| J g\|_{L^\infty(\Omega\times \mathbb S^{n-1})} \leq \|g\|_{L^\infty(\Gamma_-)}.
\]
Combining this final estimate with Lemma \ref{lem:contraction_infty} leads to the desired inequality: 

\[
\|u\|_{L^\infty}\le\sum_{j=0}^{\infty}\|LK\|^j\|Jg\|_{L^\infty}\le e^{\|\tau_-\mu_s\|_{L^\infty}}\|Jg\|_{L^\infty} \leq e^{\|\tau_-\mu_s\|_{L^\infty}}
\|g\|_{L^\infty(\Gamma_-)}.
\]

\endproof

In fact we have a more precise estimate on the pointwise value of $Jg$.\\

\begin{lemma} \label{lem:boundofJ}
Assume that $g\in  L^\infty(\Gamma_-)$. Then
\begin{eqnarray}  \label{depth}
e^{-\tau_-(x,d)\max_{\Delta_-(x,d)}\mu} |g(x-\tau_-(x,d)d, d)|
\leq |Jg(x, d)|\\ \leq e^{-\tau_-(x,d)\min_{\Delta_-(x,d)}\mu} \|g\|_{L^\infty(\Gamma_-)}, \quad \forall x\in \Omega. \nonumber
\end{eqnarray}
\end{lemma}
\proof  For $ d\in \mathbb S^n$ and $ x\in \Omega$, we have 
\begin{eqnarray*}
\left| (J g)(x, d) \right| = e^{-\int_0^{\tau_-(x,d)} \mu(x-sd) ds} \left|g(x-\tau_-(x,d)d,d)\right|.
\end{eqnarray*}
Since $\min_{\Delta_-(x,d)} \mu \leq \mu(x-sd) \leq \max_{\Delta_-(x,d)}\mu$,  the result is forward.
\endproof
It is also possible to derive an upper bound independent of the direction of the transport.
\begin{corollary}
Assume that $g\in  L^\infty(\Gamma_-)$. Then
\begin{eqnarray} \label{depth2}
  \left| (J g)(x, d) \right| \leq e^{-\textrm{\emph{dist}}(x,\partial \Omega) \min_{\Omega}\mu} \|g\|_{L^\infty(\Gamma_-)}, \quad \forall x\in \Omega.
\end{eqnarray}
\end{corollary}
\proof
The proof is a direct consequence of Lemma \ref{lem:boundofJ} and the fact that $\tau_-(x,d) \geq \textrm{dist}(x,\partial \Omega).$ 
\endproof
Notice that in the absence of the scattering source ($K=0$), the solution $u=Jg$ is exponentially decaying away from the illumination impact $\Gamma_-(d)$. Next, we extend this result to the case of the presence of scattering in the transport equation when the anisotropy factor $\eta$ approaches $1$.
\begin{theorem} \label{main1} 

Assume that $\Omega$ is a convex smooth domain, and   
let  $g\in  L^\infty(\Gamma_-)$.
Then there exists  $\eta_0 = \eta_0(\Omega, \mu_a, \mu_s) \in (0,1)$ such that, for all $\eta\in (\eta_0, 1)$,
\begin{equation} \label{assumption:contraction3}
\mathfrak t_0 =\big(1-e^{-\|\tau_- \mu_s \|_{L^\infty}}\big)  \kappa(\min_{\Omega} \mu_a, \eta) <1, 
\end{equation}
with
\begin{equation} \label{eq:kappa}
 \kappa(\alpha, \eta) \;=\; \max_{d\in \mathbb S^{n-1}, x\in \Omega} \int_{\mathbb S^{n-1}} k(d,d^\prime) e^{\alpha (\tau_-(x, d)- \tau_-(x, d^\prime))}ds(d^\prime),
  \end{equation}
  
and the unique solution $u$ to \eqref{RDT:Equation} satisfies 
\begin{equation} \label{depth}
\left\|e^{ (\min_{\Omega}\mu_a) \tau_-} u\right\|_{L^\infty(\Omega\times S^{n-1})} \leq  \frac{1}{1-\mathfrak t_0}  \|g\|_{L^\infty(\Gamma_-)}.
\end{equation}

\end{theorem}
\begin{proof}

Notice that if the condition \eqref{assumption:contraction3} is fulfilled then $LK$ becomes a contraction of the unity in  $L^\infty(\Omega \times {\mathbb S}^{n-1} )$, with norm $\|e^{ (\min_{\Omega}\mu_a) \tau_-} \cdot \|_{L^\infty(\Omega\times S^{n-1})}$, and we have 

\begin{equation} \label{etape:intermediaire}
\|e^{ (\min_{\Omega}\mu_a) \tau_-}u\|_{L^\infty}\le\sum_{j=0}^{\infty}\|LK\|^j\|e^{ (\min_{\Omega}\mu_a) \tau_-}Jg\|_{L^\infty}\le \frac{1}{1-\mathfrak t_0} \|e^{ (\min_{\Omega}\mu_a) \tau_-}Jg\|_{L^\infty}.
\end{equation}

We next derive a  bound on $\|e^{ (\min_{\Omega}\mu_a) \tau_-}Jg\|_{L^\infty}$ in terms of $ \|g\|_{L^\infty(\Gamma_-; |\nu\cdot d|)}$.  Indeed 
we have 
\begin{eqnarray*}
\left|  e^{ (\min_{\Omega}\mu_a) \tau_-(x,d)}(J g)(x, d) \right|\\ = e^{ (\min_{\Omega}\mu_a) \tau_-(x,d)} e^{-\int_0^{\tau_-(x,d)} \mu(x-sd) ds} \left|g(x-\tau_-(x,d)d,d)\right| \leq \left|g(x-\tau_-(x,d)d,d)\right|.
\end{eqnarray*}
Therefore  $\|e^{ (\min_{\Omega}\mu_a) \tau_-}Jg\|_{L^\infty} \leq \|g\|_{L^\infty(\Gamma_-)},$ 
which combined with \eqref{etape:intermediaire} yields \eqref{depth}. \\

We focus in the rest of the proof on the existence of $\eta_0 \in (0,1)$. 
\begin{lemma} \label{lemma:Lipschitz}
 Assume that $\Omega$ is a smooth star-shaped domain with respect to $x_0$, and let $\mathfrak d_\Omega$ be the diameter of the domain.\\
  Then the following estimate holds:
 \begin{equation} \label{tau:Lipschitz}
|\tau_-(x_0, d) - \tau_-(x_0, d^\prime)| \leq 2\mathfrak d_{\Omega}\|d-d^\prime\|, \qquad \forall d,\, d^\prime \in \mathbb S^{n-1}, \; d\cdot d^\prime\geq \frac{1}{2}.
\end{equation}
 
\end{lemma}
 \begin{proof}
 Since $\Omega$ is a smooth  star-shaped domain with respect to $x_0$, there exists a smooth function
 $r: \; \mathbb S^{n-1}\to [0, d_\Omega]$ such that $\partial \Omega = \left\{r(d)d: \;\; d\in  \mathbb S^{n-1} \right\}.$
 For $d, d^\prime  \in \mathbb S^{n-1}$, denote $x^\prime_0= x_0- \tau_-(x_0, d)d$ and 
 $x^{\prime \prime}_0= x_0- \tau_-(x_0, d^\prime)d^\prime$. By definition of $\tau_-$, we have $x^\prime_0, \, x^{\prime\prime}_0 \in \partial \Omega$. Moreover  $\tau_-(x_0, d)= r(d)$ and $ \tau_-(x_0, d^\prime) =
 r(d^\prime)$.  \\
 
Let $\Gamma$ be the  curve intersection of  the sector $ (x_0; \mathfrak d_{\Omega}d, 
\mathfrak d_{\Omega}d^\prime)$  with $\partial \Omega$ in between 
the points $x^{\prime}_0$ and $x_0^{\prime \prime}$.\\

Therefore 
\begin{equation} \label{size:Gamma}
|\Gamma| \;= \; \int_{d}^{d^\prime} r(v)d\tilde s(v) \leq \mathfrak d_{\Omega} \int_{d}^{d^\prime} d \tilde s(v),
\end{equation}
with  $\int_{d}^{d^\prime} ds(v)$ is the length of the  arc resulting from the intersection of the sector $ (0; d, d^\prime)$ and the unit sphere $\mathbb S^{n-1}$. Forward calculation gives 
\begin{equation} \label{size:Gamma2}
\int_{d}^{d^\prime} d\tilde s(v) = 2\arcsin(\frac{1}{2}\|d-d^\prime\|) \leq \|d-d^\prime\|,
\end{equation}
for $d\cdot d^\prime\geq \frac{1}{2}$.\\

  We deduce from the convexity of 
$\Omega$, that the segment $[x_0^{\prime };x_0^{\prime \prime} ] \subset \Omega$, and  hence 
\begin{equation*}
\left| [x_0^{\prime };x_0^{\prime \prime} ]  \right| \leq |\Gamma|. 
\end{equation*}
Consequently
\begin{equation} \label{size:Gamma3}
\left\| \tau_-(x_0, d)d-  \tau_-(x_0, d^\prime)d^\prime\right\| \leq |\Gamma|. 
\end{equation}

Combining \eqref{size:Gamma},  \eqref{size:Gamma2} and  \eqref{size:Gamma3} implies 
\[
\left\| \tau_-(x_0, d)d-  \tau_-(x_0, d^\prime)d^\prime\right\| \leq \tau_-(x_0, d^\prime)
\|d-d^\prime\|+\mathfrak d_{\Omega}
 \|d-d^\prime\|.
\]

Noticing that $\tau_-(x_0, d^\prime)\leq \mathfrak d_\Omega$, we achieve the proof of lemma.
 \end{proof}
 \begin{remark}
 Since $\Omega$ is a convex domain, it is star-shaped with respect every $x \in \Omega$, and consequently
 the estimate \eqref{tau:Lipschitz} is valid for all $x\in \Omega$. This shows that the distance $d\mapsto \tau(x, d)$
 is a Lipschitz function on the set $\mathbb S^{n-1}$, with a Lipschitz constant independent of $x$. Notice that
 $d\mapsto \tau(x, d)$ fails to be continuous in non-convex domains.
 \end{remark}

Back now to the proof of the theorem. We shall next show that $\kappa(\min_{\Omega} \mu_a, \eta)$ tends to
one when $\eta$ approaches one. 

\begin{lemma} \label{asymKappa}
Assume that $\Omega$ is a convex  smooth domain.  For $\alpha>0$ a fixed constant,  and  $\eta \in (0,1)$, define
 \begin{equation*} 
 \kappa(\alpha, \eta) \;=\; \max_{d\in \mathbb S^{n-1}, x\in \Omega} \int_{\mathbb S^{n-1}} k(d,d^\prime) e^{\alpha (\tau_-(x, d)- \tau_-(x, d^\prime))}ds(d^\prime).
  \end{equation*}
Then  $$\lim_{\eta \to 1^-} \kappa(\alpha, \eta) =1.$$
\end{lemma}
\begin{proof}
Since the kernel $k$ is normalized, we have 
 \begin{eqnarray}
 \int_{\mathbb S^{n-1}} k(d,d^\prime) e^{\alpha (\tau_-(x, d)- \tau_-(x, d^\prime))}ds(d^\prime) \nonumber \\= 1 +
 \int_{\mathbb S^{n-1}} k(d,d^\prime) \left(e^{\alpha (\tau_-(x, d)- \tau)_-(x, d^\prime))} -1\right)ds(d^\prime).
  \label{eq:split1}
  \end{eqnarray}
  We also divide the last integral  into two terms:
   \begin{eqnarray}
\int_{\mathbb S^{n-1}} k(d,d^\prime) \left(e^{\alpha (\tau_-(x, d)- \tau_-(x, d^\prime))} -1\right)ds(d^\prime)\nonumber \\
= \int_{1-d\cdot d^\prime \geq  (1-\eta)^{\frac{1}{2}}} k(d,d^\prime) \left(e^{\alpha (\tau_-(x, d)- \tau_-(x, d^\prime))} -1\right)ds(d^\prime)
\nonumber\\+ 
 \int_{1-d\cdot d^\prime < (1-\eta)^{\frac{1}{2}}} k(d,d^\prime) \left(e^{\alpha (\tau_-(x, d)- \tau_-(x, d^\prime))} -1\right)ds(d^\prime) = 
 I_1+I_2. \label{eq:split2}
 \end{eqnarray}
We first remark that 
\begin{equation} \label{eq:step1}
|I_1| \leq 
2(1+e^{2 \mathfrak d_\Omega \alpha}) ((1-\eta)^{\frac{3}{2}} +2\eta)^{-\frac{n }{2}} (1-\eta)^{1-\frac{n}{4}}.
\end{equation}
Since $|e^x - 1| \leq \max(e^x, 1)|x|$, we have 
\begin{eqnarray*} 
 |I_2| = \left|\int_{1-d\cdot d^\prime < (1-\eta)^{\frac{1}{2}}} k(d,d^\prime) \left(e^{\alpha (\tau_-(x, d)- \tau_-(x, d^\prime))}  - 1\right)
 ds(d^\prime) \right|\\ \leq \int_{1-d\cdot d^\prime < (1-\eta)^{\frac{1}{2}}}  k(d,d^\prime) e^{ \alpha \mathfrak d_\Omega} 
 \alpha |\tau_-(x, d)- \tau_-(x, d^\prime)| ds(d^\prime).
 \end{eqnarray*}
We deduce from Lemma \ref{lemma:Lipschitz}, the following bound:
\begin{eqnarray*} 
 |I_2| \leq 2 \alpha\mathfrak d_{\Omega}  e^{ \alpha \mathfrak d_\Omega}  \int_{1-d\cdot d^\prime < (1-\eta)^{\frac{1}{2}}}  k(d,d^\prime) 
\|d-d^\prime\| ds(d^\prime).
 \end{eqnarray*}
On the other hand, we have 
\[
\| d-d^\prime\|^2  \;= \; 2(1-d\cdot d^\prime).
\]
Therefore 
\begin{eqnarray} 
 |I_2| &\leq& 4 \sqrt{2} \alpha\mathfrak d_{\Omega}  e^{ \alpha \mathfrak d_\Omega}  \int_{1-d\cdot d^\prime < (1-\eta)^{\frac{1}{2}}}  k(d,d^\prime) ds(d^\prime) (1-\eta)^{\frac{1}{4}} \nonumber \\
 & \leq& 4 \sqrt{2} \alpha\mathfrak d_{\Omega}  e^{ \alpha \mathfrak d_\Omega} (1-\eta)^{\frac{1}{4}}. \label{eq:step2}
 \end{eqnarray}
Combining \eqref{eq:step1} and \eqref{eq:step2}, we get 
\begin{eqnarray*}
\int_{\mathbb S^{n-1}} k(d,d^\prime) e^{\alpha (\tau_-(x, d)- \tau(x, d^\prime))}ds(d^\prime) =  1+ O\left((1-\eta)^{\frac{1}{4}}\right), \quad \textrm{as} \quad  \eta \to 1.
\end{eqnarray*}
Since $O\left((1-\eta)^{\frac{1}{4}}\right)$ is uniform in $d\in \mathbb S^{n-1}$ and $x\in \Omega$, we obtain the
desired result. 
\end{proof}
Next, we  finally provide  the proof of the theorem. \\

Since $1-e^{-\|\tau_- \mu_s \|_{L^\infty}}<1$, we deduce from  Lemma \ref{asymKappa} that there exists 
$\eta_0 \in (0,1)$ close enough to $1$ such that $\mathfrak t_0 <1$. The rest of the proof follows from the 
convergence of the Neumann series given at the very beginning of the proof. 

\end{proof}

Theorem   \ref{main1} indicates that the photoacoustic signal  is exponentially decaying with respect to the depth $\tau_-(x, d)$ in every direction $d\in \mathbb S^{n-1}$.  We confirm this behavior numerically in Section~\ref{sec:numerical}.

    
   


\section{Asymptotic Analysis} \label{Sec:asymptotic}  
The objective of this section is  to derive the asymptotic   leading terms of the solution $u$ of \eqref{RDT:Equation}
close to the impact of the illumination on the boundary. The expression of these leading  terms will be used in a second step in Section \ref{Sec:inverse}  to solve the inverse problem. \\

For $d\in \mathbb S^{n-1}$ and for $\rho>0$ small enough, we define the set $\Omega_{d, \rho}$ as the region
of $\Omega$ where the resolution in depth is acceptable w.r.t. the direction $d$: 

\begin{equation} \label{def:Omegarho}
\Omega_{d, \rho} \;=\;  \{x \in \Omega: \; \; \tau_-(x, d) < \rho \}.  
\end{equation} 

\begin{theorem}\label{tD:asymptotic}

Let $u$ be the solution of \eqref{RDT:Equation}, and let
 \begin{eqnarray} \label{rho}
  \rho_0=  \min\left(1.79\|\mu\|_{L^\infty}^{-1},
0.6 \|\mu\|_{L^\infty}^{-2}\|g\|_{L^\infty}^{-1}\right).
\end{eqnarray}
Then, for all $d \in  \mathbb S^{n-1}$ and $x\in\Omega_{d,\rho_0}$ we have
 \begin{eqnarray}
u(x'+td,d) &=&  \left(1-\int_0^t \mu(x^\prime+s d)ds\right) g(x^\prime, d) \nonumber  \\
&&+  \int_0^t \mu_s(x'+sd) \int_{\mathbb S^{n-1}} k(d, d^\prime)g(x^\prime, d') ds(d') ds +O(\tau_-^2(x,d)), \nonumber\\
\label{D:asymptotic}
\end{eqnarray} 
 where  $O(\tau_-^2(x,d))$  is uniform in $(x,d) \in
  \Omega_{d, \rho_0}\times \mathbb S^{n-1}$, and where $x^\prime= x-\tau_-(x,d)d$ is the impact point of the illumination at the boundary.

\end{theorem}
\begin{proof}
 Assume that $\rho$ is small enough such that $\Omega_{d, \rho}$ is non-empty. Let $x \in \Omega_{d, \rho}$ be fixed.
 Next, we shall determine the asymptotic expansion of  $u(x, d)$ when $\rho$ tends to zero. In fact when $\rho$ approaches $0$,   $\tau_-(x,d)$ tends to zero.  \\
 
 \noindent Considering the low regularity of $\mu_a$ and $\mu_s$, the strategy is to derive the asymptotic expansion in terms of $\int_s^t \mu(x^\prime + r d)dr$ rather than directly in terms of $t$. Precisely, we have
 
 \begin{proposition} \label{prop1:asym}
 \begin{eqnarray*}
 \int_s^t \mu(x^\prime+r d) dr &= & O(\tau(x,d)),\\
 e^{-\int_s^t \mu(x^\prime+r d) dr} & = & 1- \int_s^t \mu(x^\prime+r d) + o(\tau(x,d)),
 \end{eqnarray*}
  as $\tau(x,d)$ tends to zero. The terms $O(\tau(x,d)) $ and $o(\tau(x,d))$ are uniform with respect to $(x,d) \in
  \Omega_{d, \rho_0}\times \mathbb S^{n-1}$ with $\rho_0 =1.79\|\mu\|_{L^\infty}^{-1}$.
 \end{proposition}
\begin{proof}
We first have 
\[
\left|\int_s^t \mu(x^\prime+r d) dr \right| \leq \|\mu\|_{L^\infty} t.
\]
On the other hand $|e^x-1-x| \leq x^2$ for $|x| \leq 1.79$. \\

Hence for 
$$  \rho_0 \leq 1.79\|\mu\|_{L^\infty}^{-1},$$ the second 
inequality of the proposition holds.

\end{proof}

\begin{proposition} \label{prop2:asym}

For fixed $\phi \in L^\infty(\Omega\times \mathbb S^{n-1})$,  and  $g\in L^\infty(\Gamma_-) $,  the following 
asymptotic expansions
\begin{eqnarray}
Jg(x^\prime+td, d) &=& \left(1-\int_0^t \mu(x^\prime+s d)ds\right) g(x^\prime, d)+ O(\tau_-^2(x,d))  \\
LK\phi(x^\prime+td, d) &=& \int_0^t K \phi(x^\prime+sd, d) ds + O(\tau_-^2(x,d)), 
\end{eqnarray}
hold  as $\tau_-(x,d)$  tends to zero. Here $ O(\tau_-^2(x,d))$ is uniform with respect to $(x,d) \in
  \Omega_{d, \rho_0}\times \mathbb S^{n-1}$ with $\rho_0$ satisfying:
\begin{eqnarray}
  \rho_0=  \min\left(1.79\|\mu\|_{L^\infty}^{-1},
0.6 \|\mu\|_{L^\infty}^{-2}\|g\|_{L^\infty}^{-1}\right).
\end{eqnarray}
\end{proposition}

\begin{proof}
We have 
\begin{eqnarray*}
(J g)(x^\prime+td, d) = e^{-\int_0^t \mu(x^\prime+s d) ds} g(x^\prime,d), 
\end{eqnarray*}
We deduce from Proposition \ref{prop1:asym} the following relation
\begin{equation} \label{one}
\left|J g(x^\prime+td, d) - \left(1-\int_0^t \mu(x^\prime+s d)ds\right) g(x^\prime, d)\right| \leq \|\mu\|_{L^\infty}  \|g\|_{L^\infty} t^2,
\end{equation}
for $0 \leq t \leq 1.79\|\mu\|_{L^\infty}^{-1}$.\\

Using the fact that $|e^x-1| \leq 3|x|$ for $|x| \leq 1.9$, we also have 

\begin{eqnarray} \label{two}
\left|LK\phi(x^\prime+td, d) - \int_0^t K \phi(x^\prime+sd, d) ds \right| \leq 3\|\mu\|_{L^\infty}^2\|\phi\|_{L^\infty} t^2,
\end{eqnarray}
for all $ 0 \leq t \leq  0.6 \|\mu\|_{L^\infty}^{-2}\|g\|_{L^\infty}^{-1}$. \\

Now regarding the inequalities \eqref{one}-\eqref{two}, taking $$ \rho_0=  \min\left(1.79\|\mu\|_{L^\infty}^{-1},
0.6 \|\mu\|_{L^\infty}^{-2}\|\phi\|_{L^\infty}^{-1}\right),$$
finishes the proof.

\end{proof}
We now focus on the proof of the theorem. Recall that $u$ is the unique solution of the integral equation
\begin{align*} 
u = L K u + J g. 
\end{align*}
We deduce from  estimates in Proposition \eqref{prop2:asym} the following asymptotic expansion:

\begin{eqnarray*}
u(x'+td,d) &=&  \left(1-\int_0^t \mu(x^\prime+s d)ds\right) g(x^\prime, d) \nonumber  \\
&&+  \int_0^t \mu_s(x'+sd) \int_{\mathbb S^{n-1}} k(d, d^\prime)g(x^\prime, d') ds(d') ds +O(\tau_-^2(x,d)),
\end{eqnarray*} 
 where  $O(\tau_-^2(x,d))$  is uniform in $(x,d) \in
  \Omega_{d, \rho_0}\times \mathbb S^{n-1}$ with $\rho_0$ satisfying \eqref{rho}.  
 
 \end{proof}
 
 \begin{corollary}\label{tH:asymptotic}
 For $g\in L^\infty(\Gamma_-) $,  we have for all $d\in\mathbb S^{n-1}$ and $x\in\Omega_{d,\rho_0}$ that
 \begin{eqnarray}
H(x, g) = \mu_a(x) \int_{\mathbb S^{n-1}}  g(x-\tau_-(x,d)d, d)   ds(d) \nonumber\\
- \mu_a(x) \int_{\mathbb S^{n-1}} \int_0^{\tau_-(x,d)} \mu(x- sd) g(x-\tau_-(x,d)d, d) ds ds(d) \nonumber  \\
+   \mu_a(x) \int_{\mathbb S^{n-1}}  \int_{\mathbb S^{n-1}} \int_0^{\tau_-(x,d)} k(d, d^\prime) g(x-\tau_-(x,d)d, d')  \mu_s(x-sd) ds ds(d) ds(d'), \nonumber\\
+O(\tau_-^2(x,d)),
\label{H:asymptotic}
\end{eqnarray} 
where  $O(\tau_-^2(x,d))$  is uniform in $(x,d) \in
  \Omega_{d, \rho_0}\times \mathbb S^{n-1}$ with $\rho_0$ satisfying \eqref{rho}.

 \end{corollary}

 \section{Inverse Problem} \label{Sec:inverse}  

The goal in this section is to recover the coefficients $\mu_a$ and $\mu_s$ from the knowledge of 
$H(\cdot, g_j)|_{\Omega}, \, j=1, \dots, M,$ with $g_j \in L^\infty(\Gamma_-)$. The approach 
is based on the asymptotic expansion \eqref{H:asymptotic} derived in the previous section. Indeed, in 
the first step  we shall neglect the remainder $o(\tau_-^2(x,d))$, and consider only the three first terms in the 
asymptotic  expansion of the absorbed optical energy density $H(x,g)$.  The second step is to linearize
the system by factorizing  the term on the right-hand side of the equation by $\mu_a$, and then dividing
 the entire equation by the factor of  $\mu_a$. \\
 Set
\begin{eqnarray*}
m_0(x, g)=  \int_{\mathbb S^{n-1}}  g(x-\tau_-(x,d)d, d)   ds(d), \\
m_1(x, g) =   -\int_{\mathbb S^{n-1}} \int_0^{\tau_-(x,d)} \mu(x- sd) g(x-\tau_-(x,d)d, d) ds ds(d),\\
m_2(x,g) =   \int_{(\mathbb S^{n-1})^2} \hspace{-1mm} \int_0^{\tau_-(x,d)} \hspace{-2mm}k(d, d^\prime) g(x-\tau_-(x,d)d, d')  
 \mu_s(x-sd) ds ds(d) ds(d'). 
 \end{eqnarray*}

Linearizing \eqref{H:asymptotic}, we deduce the following approximation:
\begin{eqnarray}
\mu_a(x) \approx H(x,g) \left( \sum_{j=0}^2 m_j(x,g) \right)^{-1} \nonumber\\
\approx  m_0^{-1}(x,g)H(x,g)- m_0^{-2}(x,g)H(x,g)m_1(x,g)-m_0^{-2}(x,g)H(x,g)m_2(x,g). 
\label{approx1}
 \end{eqnarray}
Notice that since the illumination $g \geq 0$ is not trivial, the function $m_0(x,g)$ is non-zero.
The equation \eqref{approx1} is linear in the unknowns $(\mu_a, \mu_s)$,  and can be rewritten in the
following form: 
\begin{equation} \label{newintegral}
(I-T_{1,g})\mu_a +T_{2,g} \mu_s \approx  m_0^{-1}(x,g)H(x,g), \qquad  x\in \Omega_{\rho_0},
\end{equation}  
with $\Omega_{\rho_0} =  \cup_{d\in \mathbb S^{n-1}} \Omega_{d, \rho_0}$,
where $T_{k,g} : L^\infty(\Omega_{\rho_0}) \to L^\infty(\Omega_{\rho_0}), \, k=1, 2,$ are linear  integral operators: 
\begin{eqnarray}
T_{1,g} w =  m_0^{-2}(x,g)H(x,g) \int_{\mathbb S^{n-1}} \int_0^{\tau_-(x,d)} w(x- sd) g(x-\tau_-(x,d)d, d) ds ds(d), \label{T_1}\\
T_{2,g} w =  -T_{1,g} w  \nonumber \\ + m_0^{-2}(x,g)H(x,g) \hspace{-2mm}  \int_{(\mathbb S^{n-1})^2} \hspace{-2mm} \int_0^{\tau_-(x,d)} \hspace{-6mm}k(d, d^\prime) g(x-\tau_-(x,d)d, d')  
 w(x-sd) ds ds(d) ds(d').
\end{eqnarray}

The inverse problem  we consider here is to recover the coefficients  $(\mu_a, \mu_s)$, solution to the 
system 

\begin{equation} \label{newintegral2}
(I-T_{1,g_k})\mu_a +T_{2,g_k} \mu_s =  m_0^{-1}(x,g_k)H(x,g_k), \qquad  x\in \Omega_{\rho_0}, \; k=1, \cdots, M.
\end{equation}  

\section{Numerical results}\label{sec:numerical}

In this section, we illustrate numerically our results regarding the exponential depth-decay of the RTE solution $u$, the asymptotic approximation of $u$ and of the internal heat data $H$ near the illumination impact, as well as the solution $(\mu_a,\mu_s)$ of the linearized inverse problem in Eq.~\eqref{newintegral2}. Figure~\ref{fig:gtcig} shows the
problem setup.
\begin{figure}
    \centering
    \includegraphics[scale=0.5]{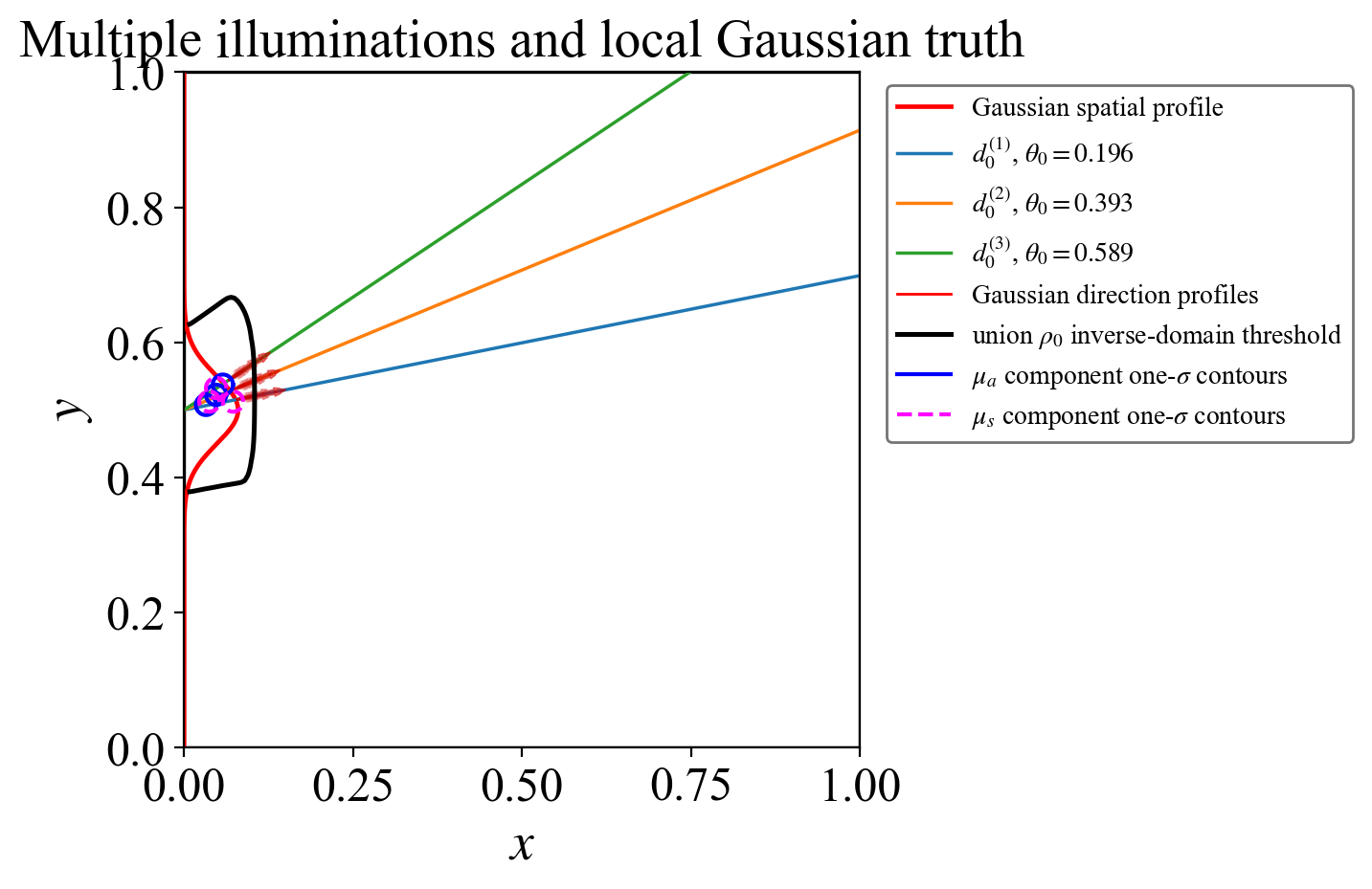}\\
    \includegraphics[scale=0.5]{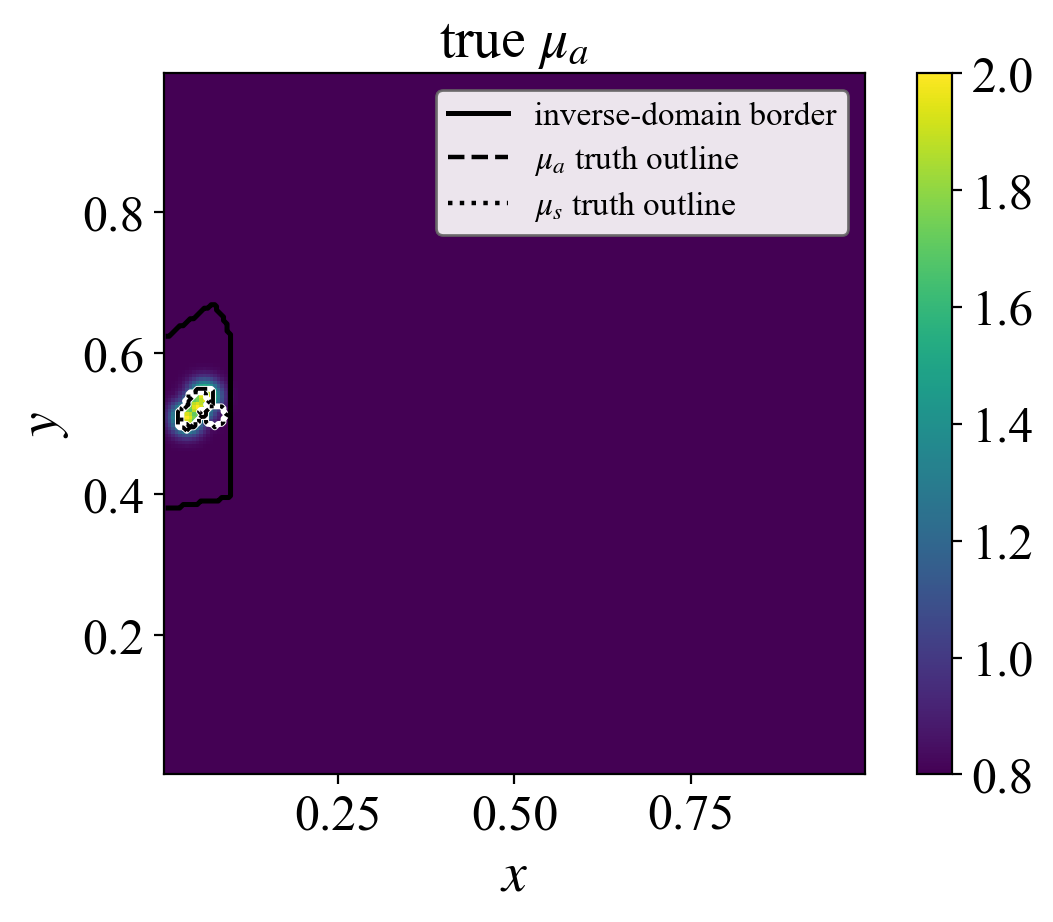}\includegraphics[scale=0.5]{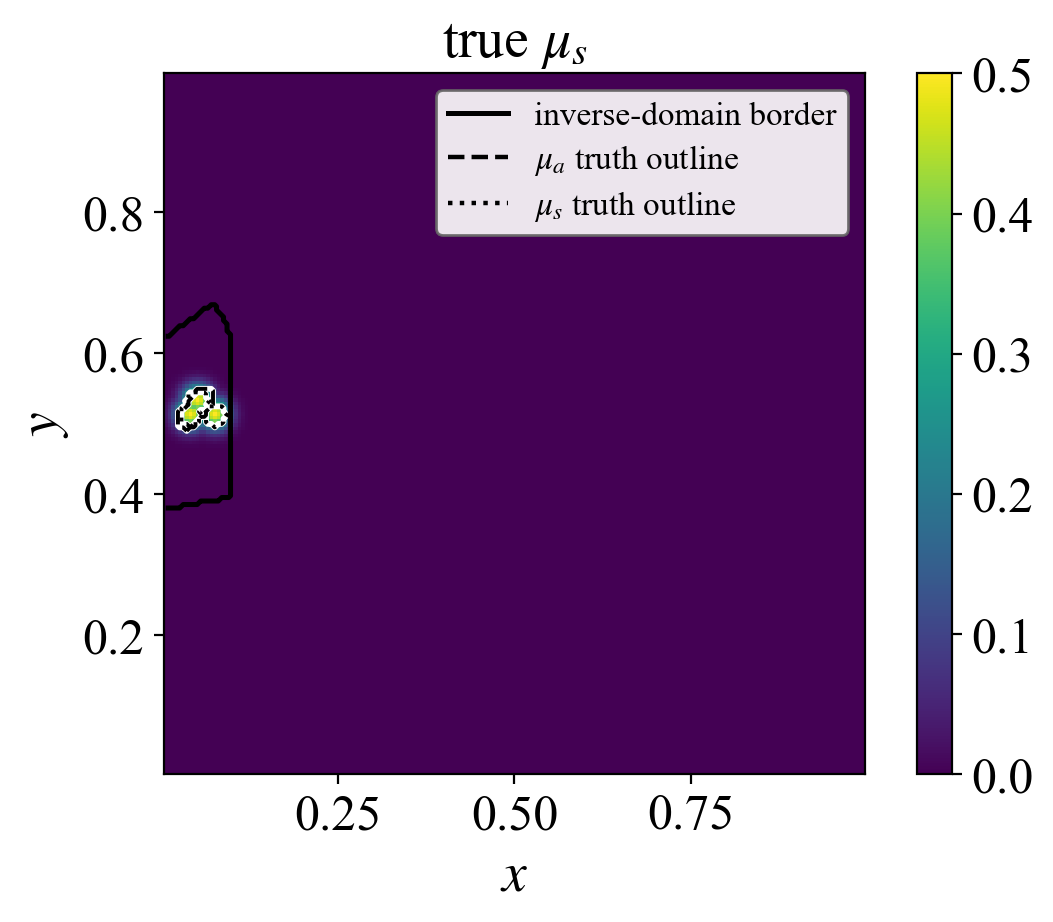}
    \caption{Top: three Gaussian boundary illuminations and the resulting local \(\rho_0\)-visible inverse region. The component outlines show the Gaussian truth components for \(\mu_a\) and \(\mu_s\). Bottom: ground-truth absorption and scattering coefficients.}
    \label{fig:gtcig}
\end{figure}
We let the domain be the unit square, $\Omega=(0,1)\times(0,1)$, with illumination supported on the left boundary. The incoming boundary data are localized both in the boundary variable and in direction. For
\(d(\theta)=(\cos\theta,\sin\theta)\), the \(j\)-th illumination is given by
\[
g_j((0,y),d(\theta))
=
A
\exp\!\left(
-\frac{(y-y_c)^2}{2\sigma_y^2}
\right)
\exp\!\left(
-\frac{\operatorname{ang}(\theta-\theta_{0,j})^2}{2\sigma_\theta^2}
\right)
\mathbf{1}_{\{\cos\theta>0\}},
\]
where '\(\operatorname{ang}\)' is the wrapped angular difference. In the
computations below, we set $A=1$, $y_c=0.5$, $\sigma_y=0.05$, and $\sigma_\theta=0.1$. Thus the beam is centered at the midpoint of the illuminated boundary and is
directed into \(\Omega\). We use three illuminations in the inverse problem, with central angles $\theta_{0,1}=\pi/16$, $\theta_{0,2}=\pi/8$, and $\theta_{0,3}=3\pi/16$, corresponding to the central-ray directions $d_{0,1}\approx(0.9808,0.1951)$, $d_{0,2}\approx(0.9239,0.3827)$, and $d_{0,3}\approx(0.8315,0.5556)$.
The second illumination, \(\theta_{0,2}=\pi/8\), is used as the reference illumination for the one-dimensional plots along the central ray. This ray starts at \(x_0=(0,0.5)\), exits through the right boundary, and has exit distance $\rho_{\mathrm{exit}}=1/\cos(\pi/8)\approx 1.0824$. The optical coefficients are chosen to be localized near the illuminated part
of the boundary, in the region where the local asymptotic approximation is
valid. Specifically, the ground truth is generated as a collection of three
small Gaussian components for each coefficient:
\[
\mu_a(x)
=
\mu_{a,\mathrm{bg}}
+
\bigl(\mu_{a,\mathrm{peak}}-\mu_{a,\mathrm{bg}}\bigr)
\max_{\ell=1,2,3}\phi^a_\ell(x),
\]
and
\[
\mu_s(x)
=
\mu_{s,\mathrm{bg}}
+
\bigl(\mu_{s,\mathrm{peak}}-\mu_{s,\mathrm{bg}}\bigr)
\max_{\ell=1,2,3}\phi^s_\ell(x).
\]
Here each component is an isotropic Gaussian normalized on the computational
grid \(\Omega_h\):
\[
\phi^a_\ell(x)
=
\frac{
\exp\!\left(-|x-c^a_\ell|^2/(2\sigma_a^2)\right)
}{
\max_{z\in\Omega_h}
\exp\!\left(-|z-c^a_\ell|^2/(2\sigma_a^2)\right)
},
\]
\[
\phi^s_\ell(x)
=
\frac{
\exp\!\left(-|x-c^s_\ell|^2/(2\sigma_s^2)\right)
}{
\max_{z\in\Omega_h}
\exp\!\left(-|z-c^s_\ell|^2/(2\sigma_s^2)\right)
}.
\]
We use $\mu_{a,\mathrm{bg}}=0.8$, $\mu_{a,\mathrm{peak}}=2.0$,
and $\mu_{s,\mathrm{bg}}=0$, $\mu_{s,\mathrm{peak}}=0.5$. The Gaussian widths are $\sigma_a=\sigma_s=0.015$. The component centers \(c^a_\ell\) are placed inside the local
\(\rho_0\)-visible inverse region associated with the three illuminations, at depths $0.36\rho_0$, $0.55\rho_0$, and $0.74\rho_0$
along the central illumination directions. Similarly, the centers \(c^s_\ell\) are placed at $0.40\rho_0$, $0.59\rho_0$, and $0.78\rho_0$ along the central illumination directions. For the present parameter values, all one-standard-deviation neighbourhoods
of the Gaussian truth components lie inside the active inverse region on the computational grid. The depth bound used for the validity of the asymptotic expression of Section~3 is
\[
\rho_0
=
\min\left\{
1.79\,\|\mu\|_{L^\infty(\Omega)}^{-1},
\;
0.6\,\|\mu\|_{L^\infty(\Omega)}^{-2}
\|g\|_{L^\infty(\Gamma_-)}^{-1}
\right\}=0.096.
\]
We solve the inverse problem only in the local region determined by this value
of \(\rho_0\). For the \(j\)-th illumination, define
\[
m_{0,j}(x)
=
\int_{\mathbb S^1}
g_j\!\left(x-\tau_-(x,d)d,\ d\right)\,ds(d).
\]
We also define the normalized \(\rho_0\)-reach function
\[
Q_j(x)
=
\frac{
\displaystyle
\int_{\mathbb S^1}
\mathbf{1}_{\{\tau_-(x,d)\leq \rho_0\}}
g_j\!\left(x-\tau_-(x,d)d,\ d\right)\,ds(d)
}{
\displaystyle
\max_{z\in\Omega_h}
\int_{\mathbb S^1}
\mathbf{1}_{\{\tau_-(z,d)\leq \rho_0\}}
g_j\!\left(z-\tau_-(z,d)d,\ d\right)\,ds(d)
}.
\]
The active inverse region for illumination \(j\) is then
\[
R_j=\left\{x\in\Omega_h:\ Q_j(x)\geq 0.05\right\}\cap\{x\in\Omega_h:\ m_{0,j}(x)\geq 0.01\max_{z\in\Omega_h}m_{0,j}(z)\},
\]
and the reconstruction region is $R=\bigcup_{j=1}^3 R_j$. Outside \(R\), the coefficients are fixed at their background values. Thus the
unknowns are introduced only where the incoming illumination is sufficiently
strong and where the \(\rho_0\)-scale asymptotic model is used. The angular variable is discretized with \(N_d=64\) uniformly spaced
directions,
\[
    d_m=(\cos\theta_m,\sin\theta_m),
    \qquad
    \Delta\theta=\frac{2\pi}{N_d}.
\]
The spatial grid used for the figures has \(N_x=N_y=100\). The scattering phase function is the two-dimensional Henyey--Greenstein kernel, discretized as a matrix normalized by
\[
    \Delta\theta
    \sum_{m'=1}^{N_d}
    k(d_m,d_{m'})
    =
    1,
    \qquad m=1,\ldots,N_d.
\]
For the reconstruction figures we take the anisotropy parameter \(\eta=0.6\). Figure~\ref{fig:t0eta} shows \(\mathfrak t_0\) as a function of \(\eta\) in the present case.
\begin{figure}
    \centering
    \includegraphics[scale=0.5]{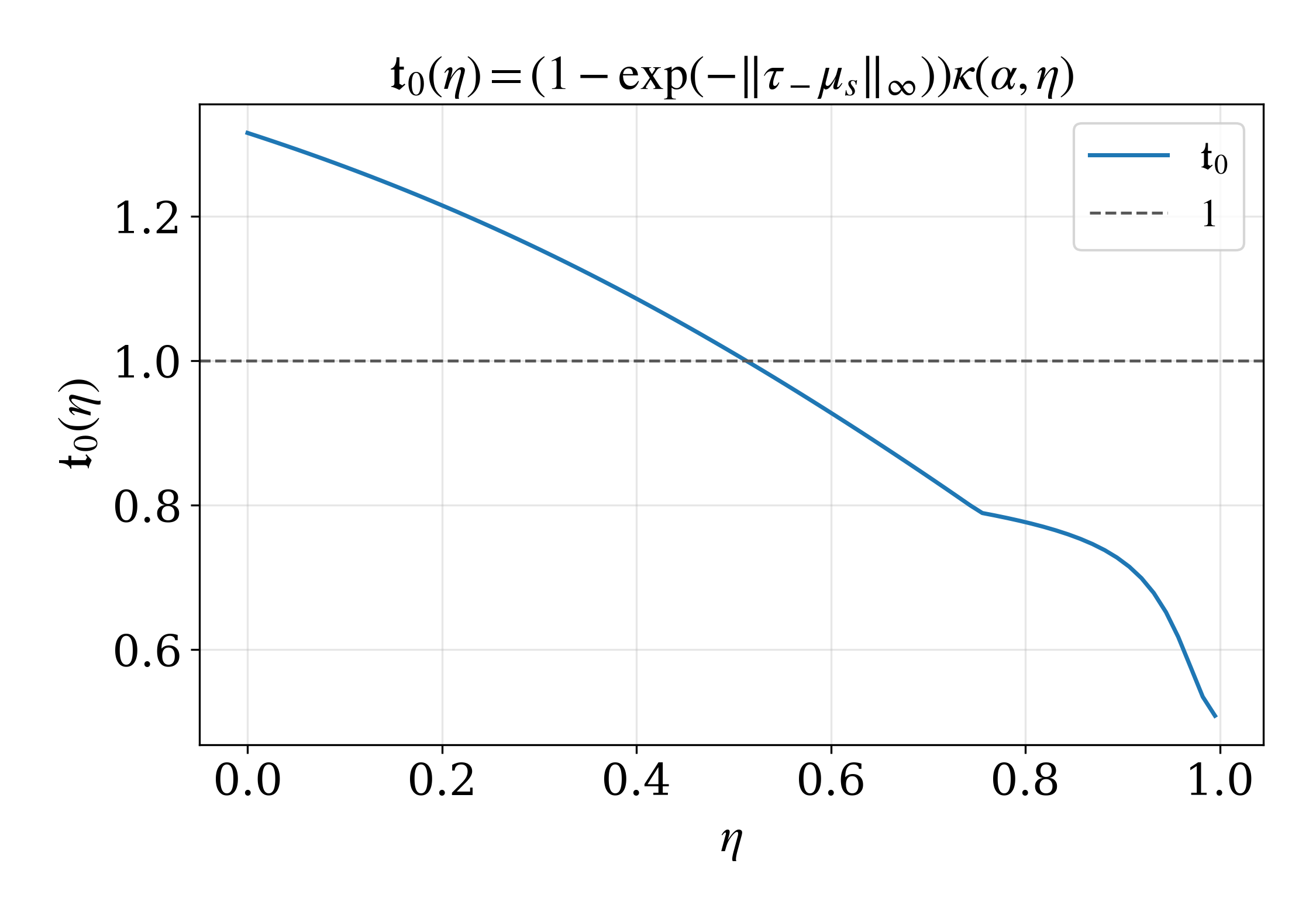}
    \caption{Illustration of the existence of \(\eta_0\) of Theorem~\ref{main1}. The value $\eta=0.6$, chosen for the numerical experiment, is in the region $\mathfrak t_0(\eta)<1$ where the RTE solution $u$ decays exponentially with depth.}
    \label{fig:t0eta}
\end{figure}
It demonstrates numerically the existence of \(\eta_0\) in
Theorem~\ref{main1}, such that \(\mathfrak t_0(\eta)<1\) for all
\(\eta>\eta_0\), and such that the RTE solution $u$ is guaranteed to decay exponentially with depth. Next, Figure~\ref{fig:uexp} shows the
exponential decay of the directional full-RTE field $u(x_0+\rho d_{0,2},d_{0,2})$ along the reference central ray.
\begin{figure}
    \centering
    \includegraphics[scale=0.5]{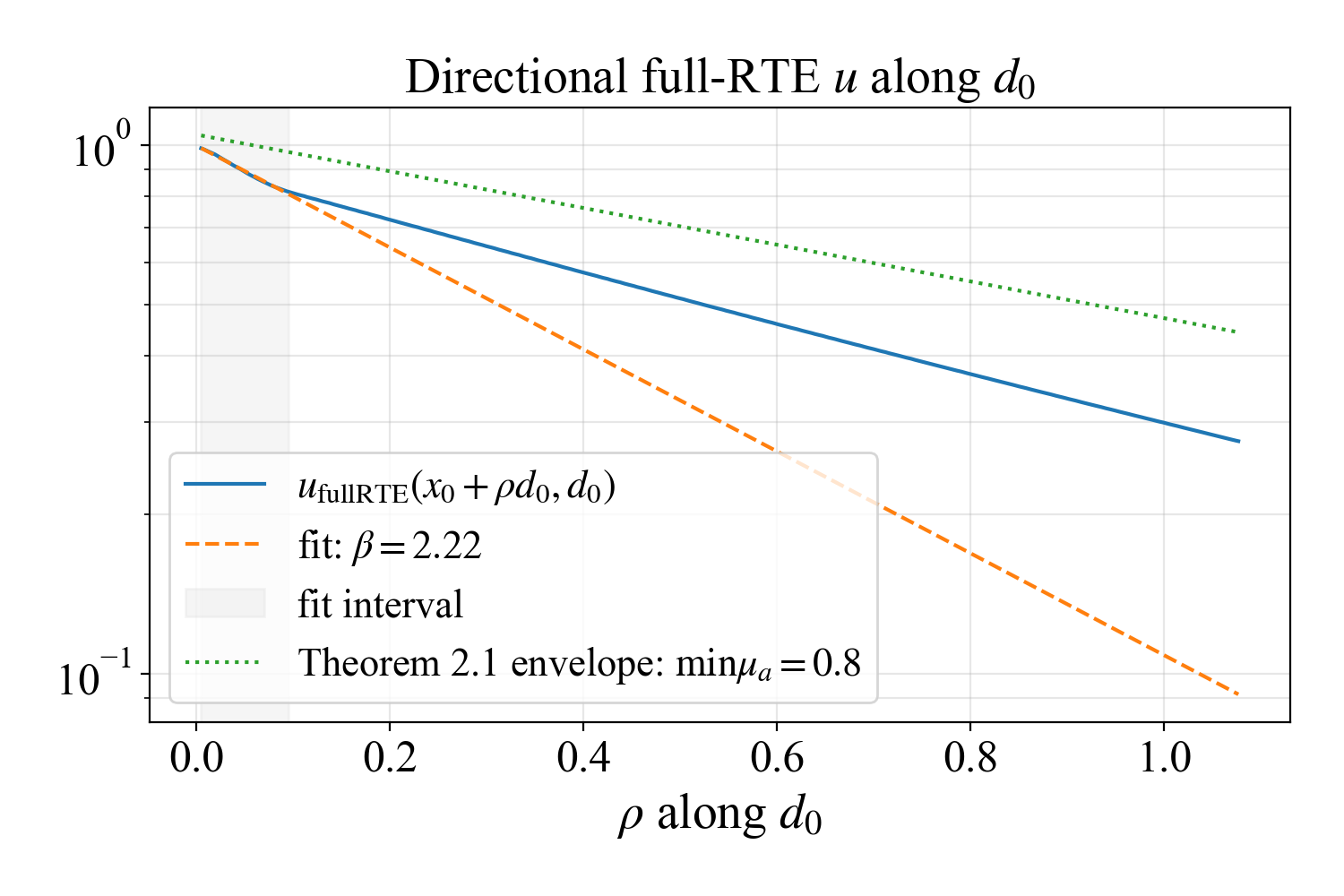}
    \caption{Exponential decay of the directional full-RTE field \(u\) along the reference central ray.}
    \label{fig:uexp}
\end{figure}

To compute the full RTE solution, we implemented a solver in the class of upwind-difference, discrete-ordinates transport methods used in optical tomography~\cite{Klose}, using source iteration with prescribed incoming boundary illumination. Recall that \(\mu=\mu_a+\mu_s\). The discrete phase
matrix is row-normalized as above. The source iteration is initialized by the
ballistic component, that is, the solution obtained by omitting the scattering
source. In continuous form this component is
\[
u_{\mathrm{bal}}(x,d)
=
g\!\left(x-\tau_-(x,d)d,d\right)
\exp\!\left(
-\int_0^{\tau_-(x,d)}
\mu(x-sd)\,ds
\right).
\]
In the implementation, \(u_{\mathrm{bal}}\) is approximated on the Cartesian
grid by solving
\[
d_m\cdot\nabla u_{\mathrm{bal}}(\cdot,d_m)
+
\mu u_{\mathrm{bal}}(\cdot,d_m)
=
0
\]
for each discrete direction \(d_m\), using first-order upwind finite
differences with the prescribed incoming boundary data. Each source iteration
then computes
\[
d_m\cdot\nabla u_m^{(\ell+1)}(x)
+
\mu(x)u_m^{(\ell+1)}(x)
=
\mu_s(x)\,
\Delta\theta
\sum_{m'=1}^{N_d}
k(d_m,d_{m'})
u_{m'}^{(\ell)}(x),
\]
again with the prescribed incoming boundary values. The iteration is
terminated when
\[
\frac{
\left(
\displaystyle\sum_{m=1}^{N_d}
\left\|
u_m^{(\ell+1)}-u_m^{(\ell)}
\right\|_{\ell^2_h(\Omega)}^2
\right)^{1/2}
}{
\left(
\displaystyle\sum_{m=1}^{N_d}
\left\|
u_m^{(\ell+1)}
\right\|_{\ell^2_h(\Omega)}^2
\right)^{1/2}
}
<
\varepsilon_{\mathrm{RTE}},
\qquad
\varepsilon_{\mathrm{RTE}}=10^{-7}.
\]

Next, in figures~\ref{fig:u_asymptotic_validation} and~\ref{fig:H_asymptotic_validation} we investigate the validity of the asymptotic approximation of the RTE solution $u$ and of the
internal heat map from Theorem~\ref{tD:asymptotic} and Corollary~\ref{tH:asymptotic}, respectively. The comparison with the
full-RTE solution is used only as a diagnostic test of the forward
asymptotics. The inverse reconstruction below is based on the asymptotic datum \(H_{\mathrm{asymp}}\), not on the full-RTE internal datum.
\begin{figure}
    \centering
    \includegraphics[width=0.49\linewidth]{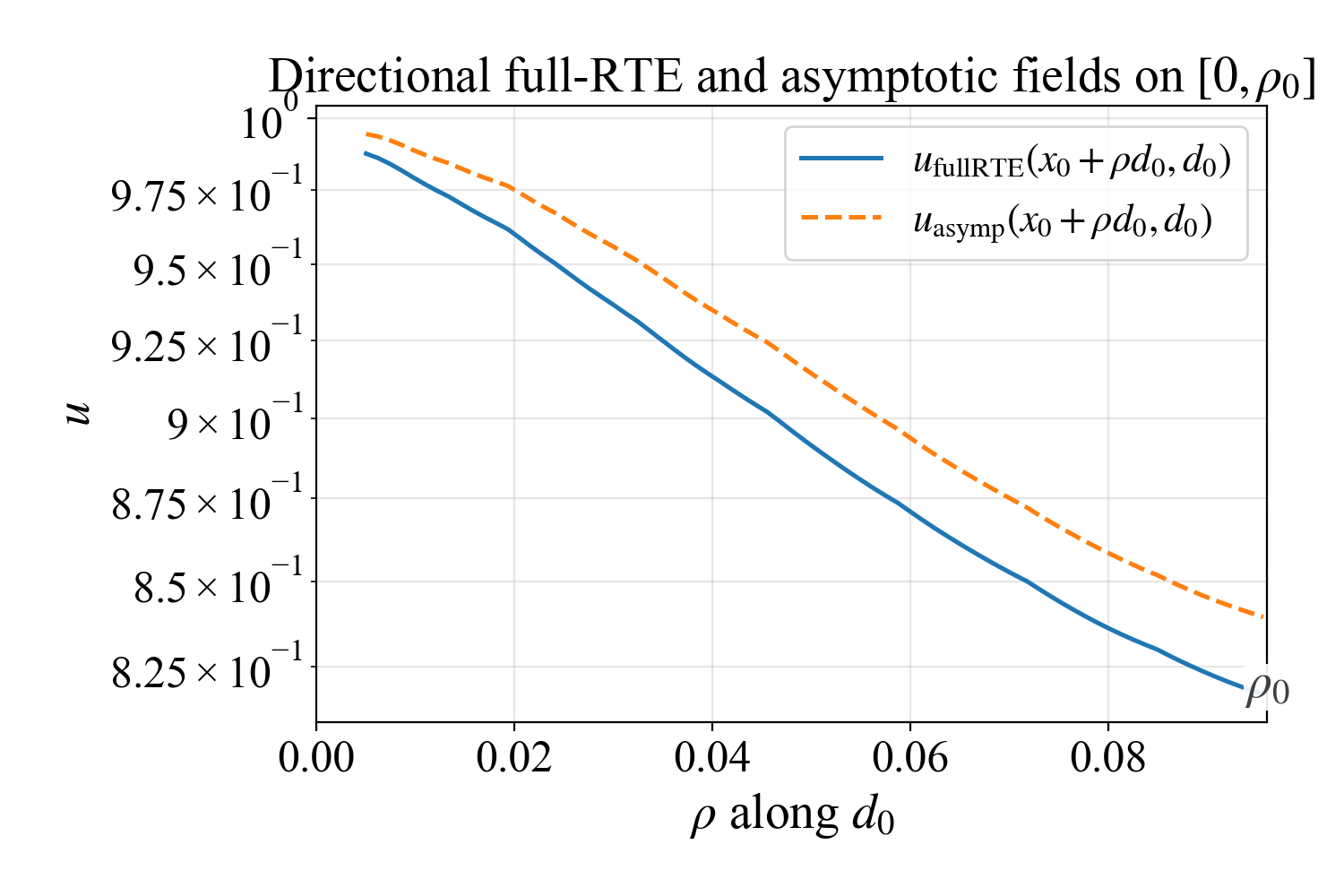}
    \includegraphics[width=0.49\linewidth]{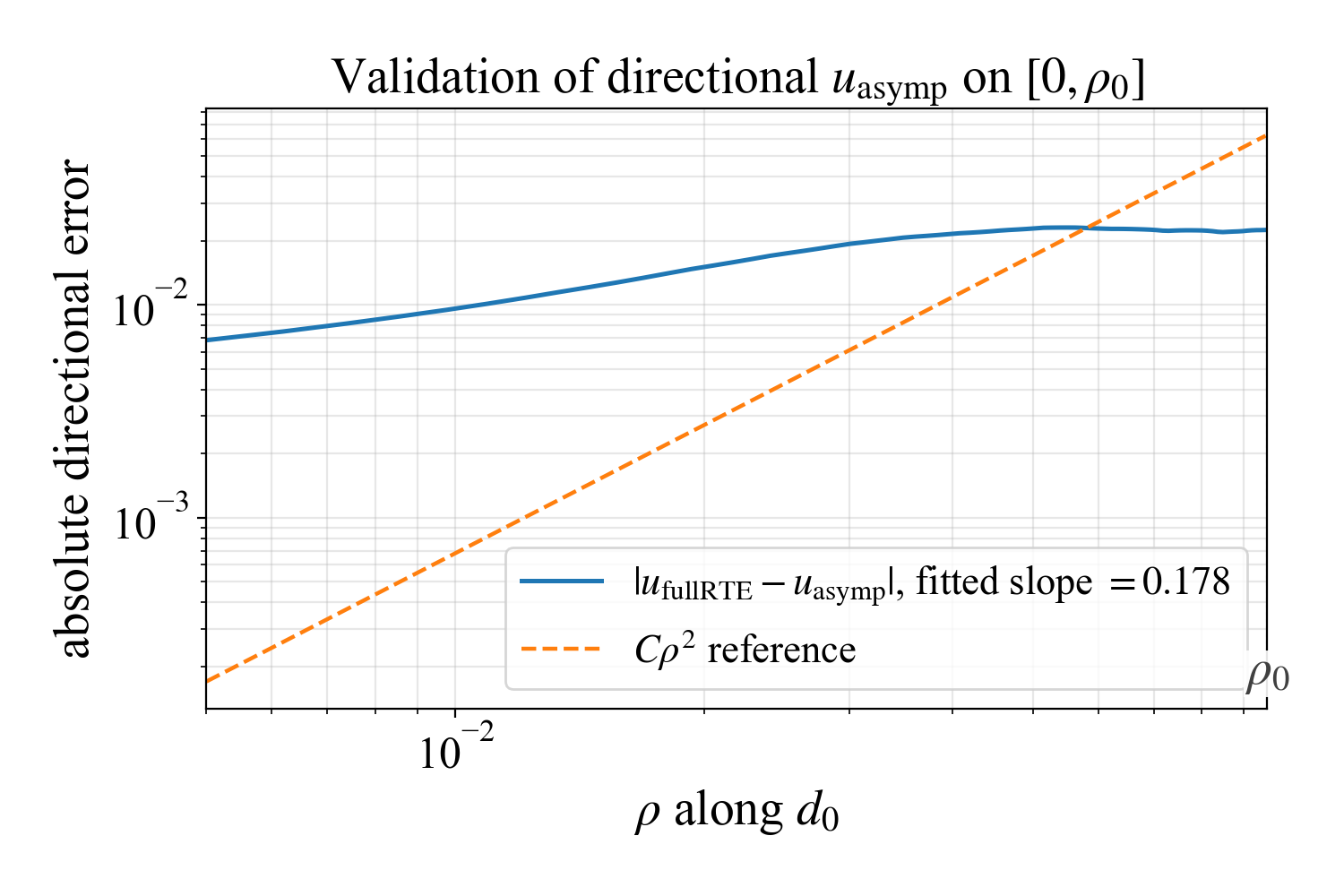}
    \caption{Directional field along the reference central ray. Left: full-RTE field and local asymptotic approximation. Right: absolute asymptotic error as a function of depth.}
    \label{fig:u_asymptotic_validation}
\end{figure}
\begin{figure}
    \centering
    \includegraphics[width=0.49\linewidth]{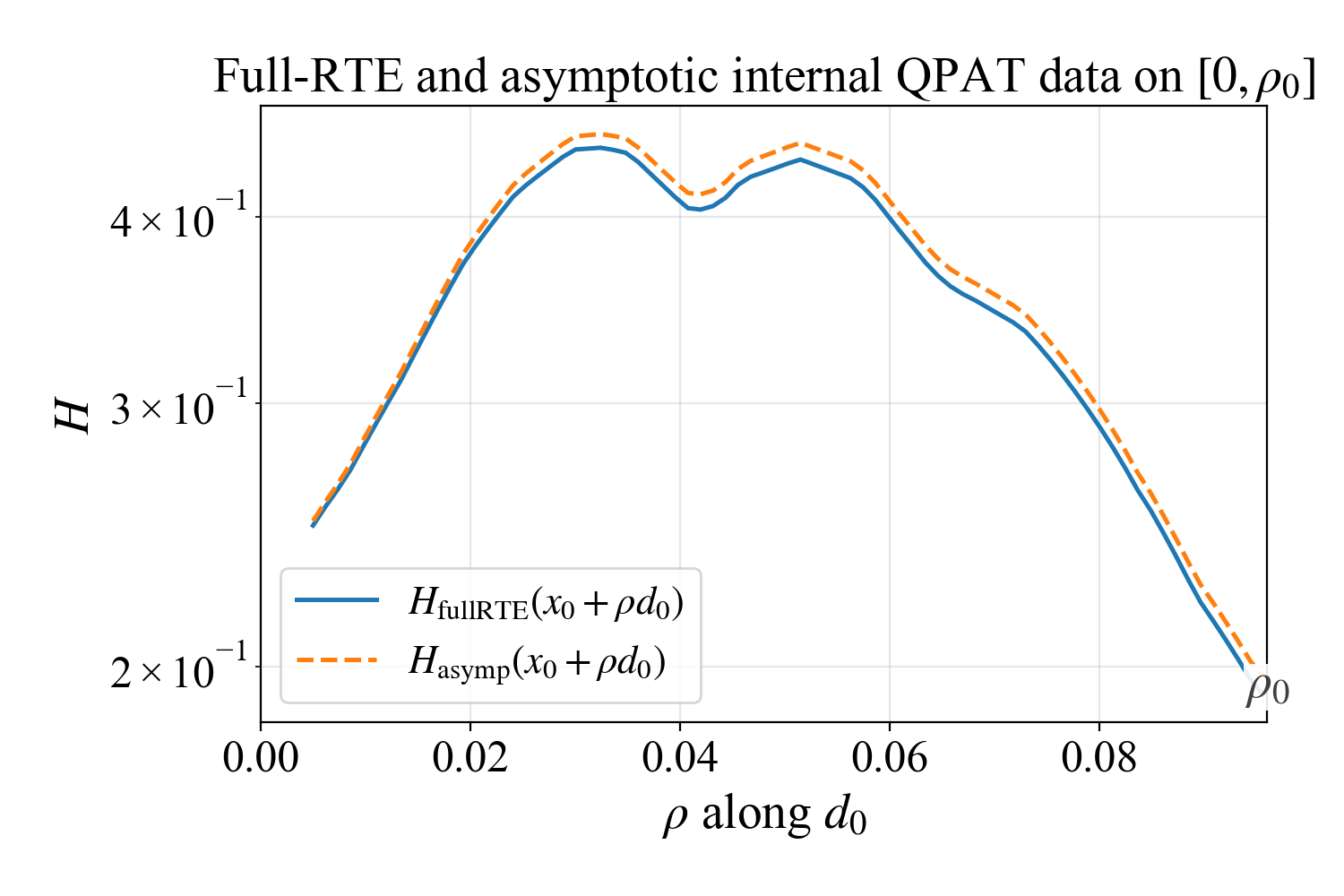}
    \includegraphics[width=0.49\linewidth]{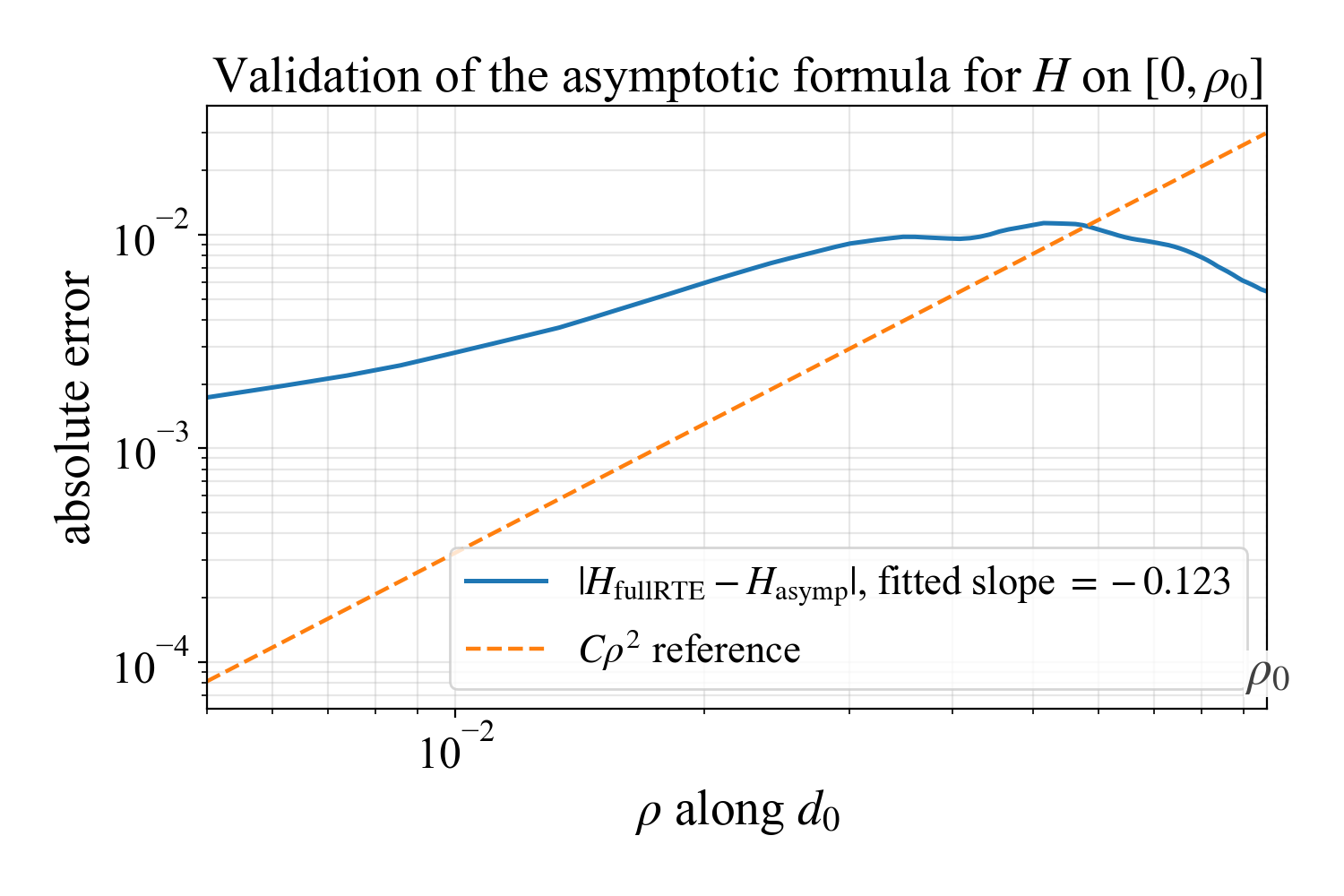}
    \caption{Internal heat data along the reference central ray. Left: full-RTE internal datum and asymptotic internal datum. Right: absolute asymptotic error as a function of depth.}
    \label{fig:H_asymptotic_validation}
\end{figure}
Both asymptotic expansions perform well within the $\rho_0$-defined thin layer close to the illumination impact. The slopes of the residuals are proportional to $\rho^p$, with $0<p<2$, as $\rho\rightarrow0$, while the residuals are of the order $\rho^2$ in Theorem~\ref{tD:asymptotic} and Corollary~\ref{tH:asymptotic}. However, this is not unexpected, since the comparison combines the asymptotic truncation error with discretization error from the first-order upwind transport solver and finite-grid sampling near the boundary.

Finally, we solve the linearized inverse problem in Eq.~\eqref{newintegral2}
on the local reconstruction region \(R\). The absorption coefficient is reconstructed pixelwise on \(R\). The scattering coefficient is represented in terms of a finite linear combination of Gaussians,
\[
    \mu_s(x) = \mu_{s,\mathrm{bg}} + \sum_{q=1}^{Q} c_q b_q(x),
    \qquad c_q\geq 0,
\]
whose centers are placed on a Cartesian lattice covering the
reconstruction region. In the present computation, the candidate lattice has size \(11\times31\), and
the Gaussian widths are $\sigma_{b,x}=\sigma_{b,y}=0.015$. The equations for the three illuminations are stacked. In continuous notation, the resulting least-squares problem is
\begin{equation}\label{lin}
\min_{\mu_a,\{c_q\}}\sum_{j=1}^{3}\left\|(I-T_{1,g_j})\mu_a+
T_{2,g_j}\left(\mu_{s,\mathrm{bg}}+\sum_{q=1}^{Q} c_q b_q
\right)-m_{0,j}^{-1}H_j\right\|_{L^2(R_j)}^2+\mathcal R(\mu_a,\{c_q\}),
\end{equation}
where \(\mathcal R\) denotes the small zeroth-order and gradient
regularization terms used for numerical stability. The coefficient bounds are
imposed during the solve.

The forward full-RTE computation on the \(100\times100\times64\) grid takes
approximately 11 minutes. The inverse problem itself is solved only on the
small local reconstruction region and takes approximately 5 seconds on a
standard desktop. The reconstruction results are shown in
Figures~\ref{fig:mua_reconstruction} and~\ref{fig:mus_reconstruction}. The pointwise relative error in the reconstruction of $\mu_a$ is on the order of $2\%$. As
expected from Eq.~\eqref{newintegral2}, the reconstruction of \(\mu_s\) is
more ill-conditioned than that of \(\mu_a\). Indeed, \(\mu_a\) enters the linearized equation through the leading identity term, whereas \(\mu_s\) enters only through the integral operator \(T_{2,g_j}\). We finally remark that, although the linearization~\eqref{approx1} introduces another approximation into the forward model, our numerical investigation has shown that it also significantly stabilizes the  solution of the inverse problem. In fact, we obtained better reconstructions of the coefficients $\mu_a$ and $\mu_s$ by minimizing the linearized functional in~\eqref{lin} than by solving the nonlinear optimization problem associated with the original asymptotic forward model.

\begin{figure}
\begin{center}
    \includegraphics[scale=0.37]{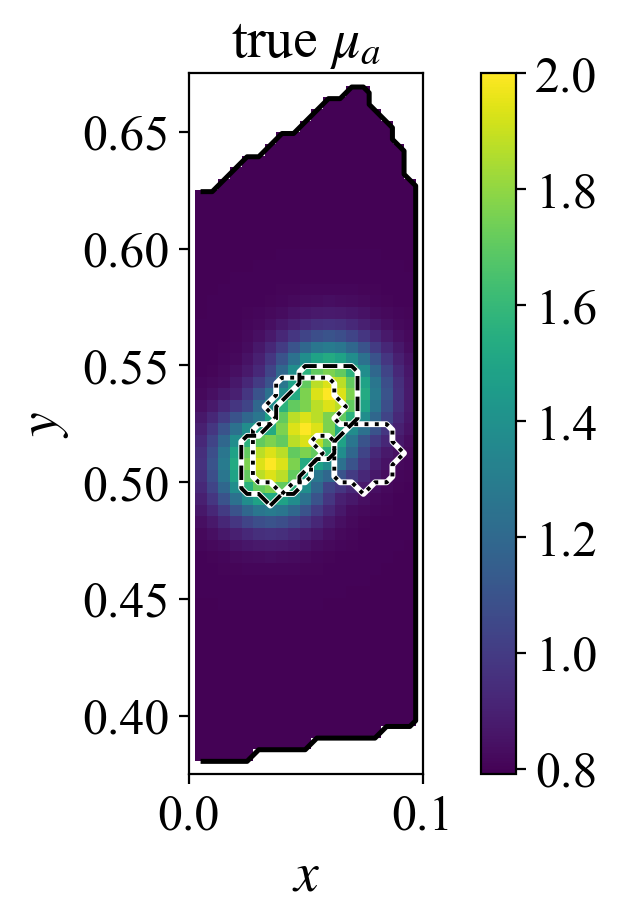}
    \includegraphics[scale=0.37]{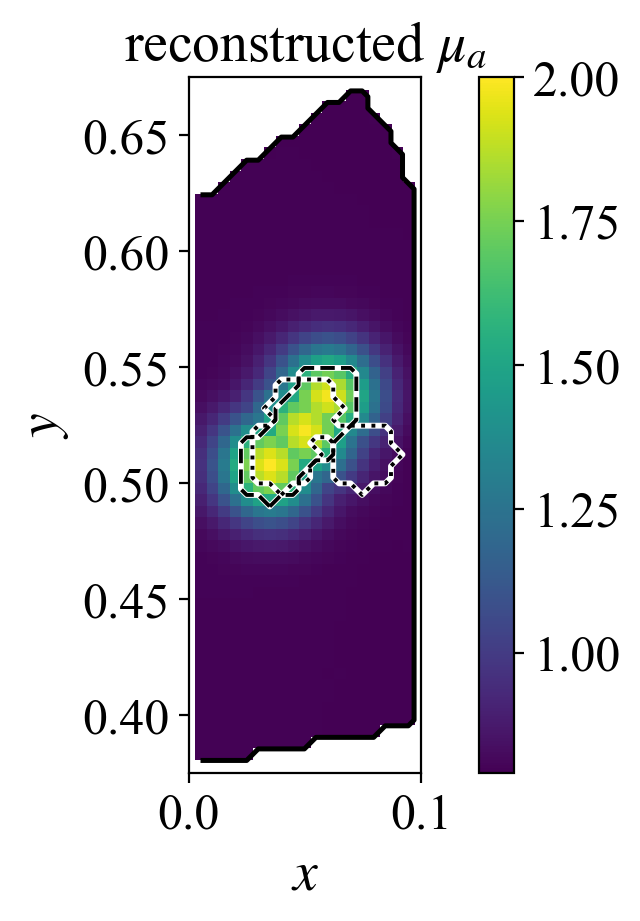}
    \includegraphics[scale=0.37]{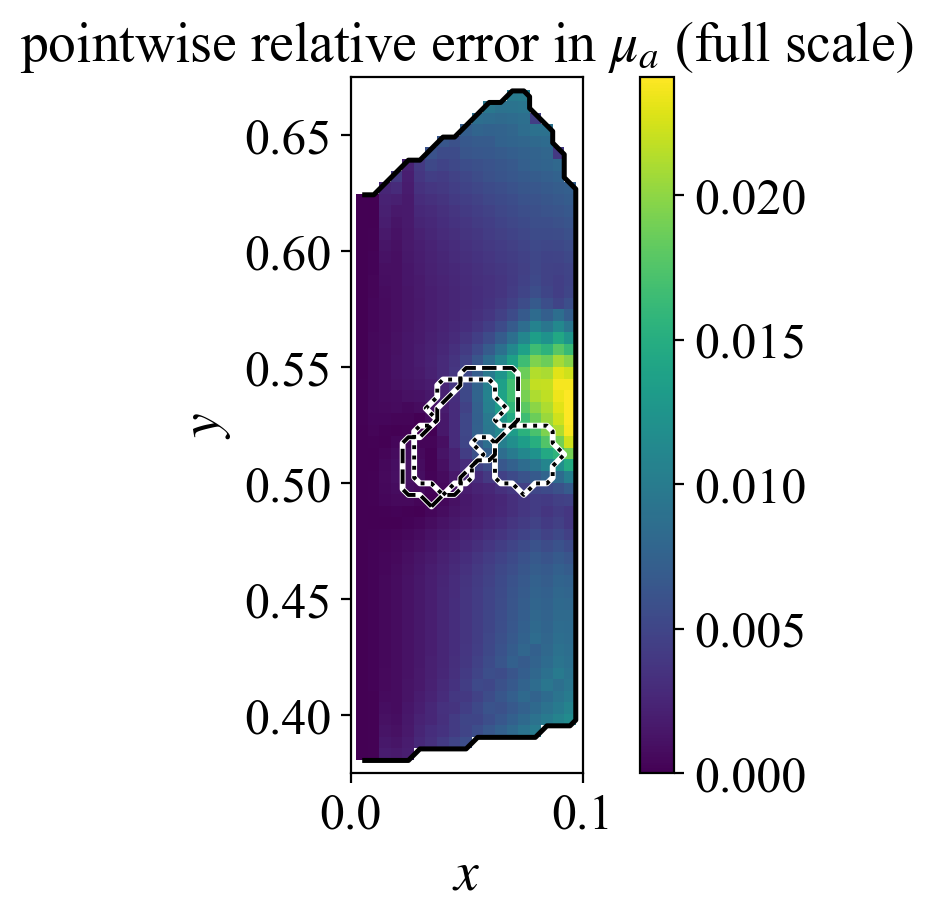}
    \includegraphics[scale=0.35]{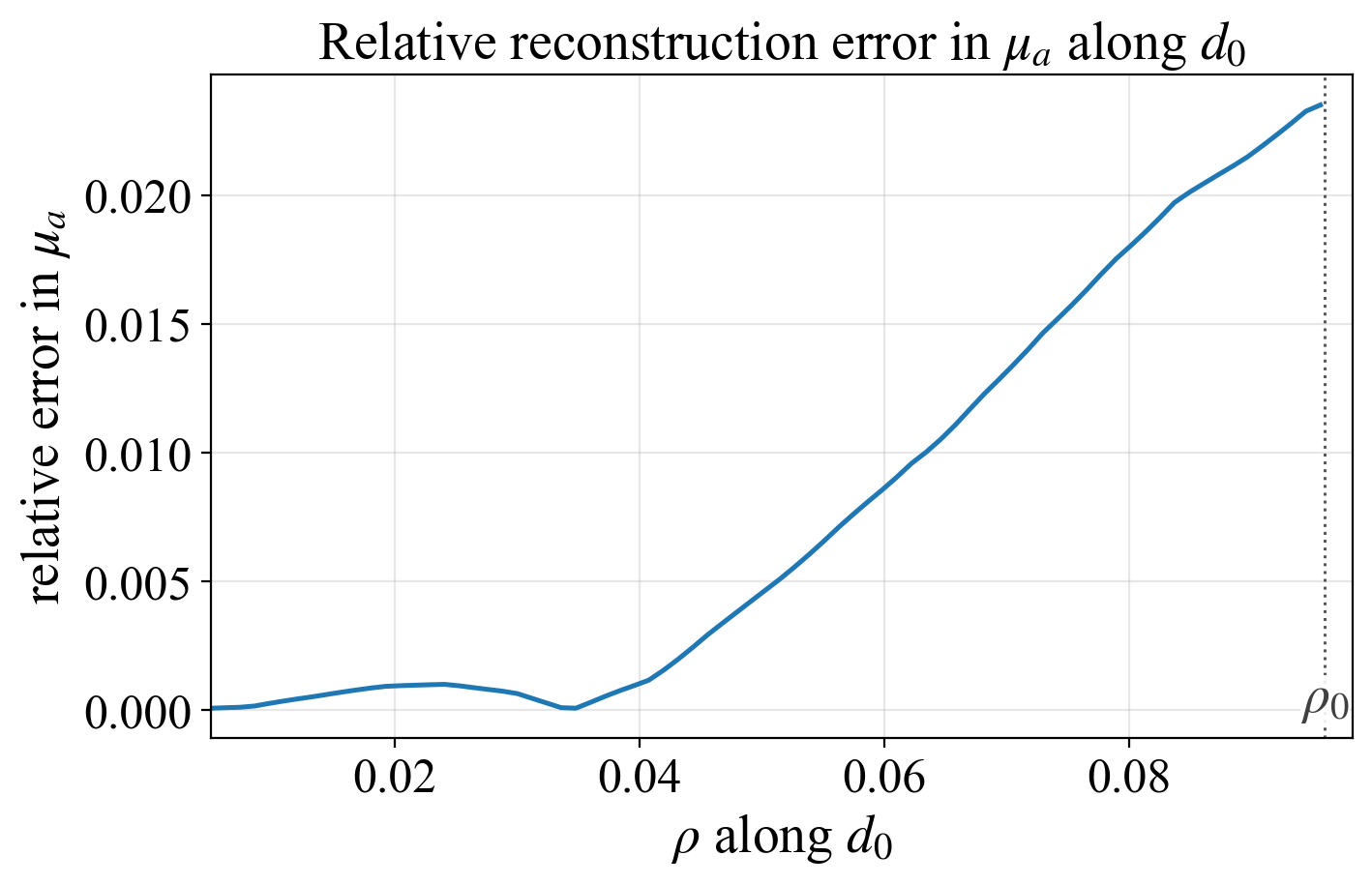}
\end{center}
\caption{Reconstruction of \(\mu_a\) for \(\eta=0.6\). From left to right: ground truth in the reconstruction window, reconstruction, relative error, and relative error along the reference central ray.}
\label{fig:mua_reconstruction}
\end{figure}

\begin{figure}
\begin{center}
    \includegraphics[scale=0.37]{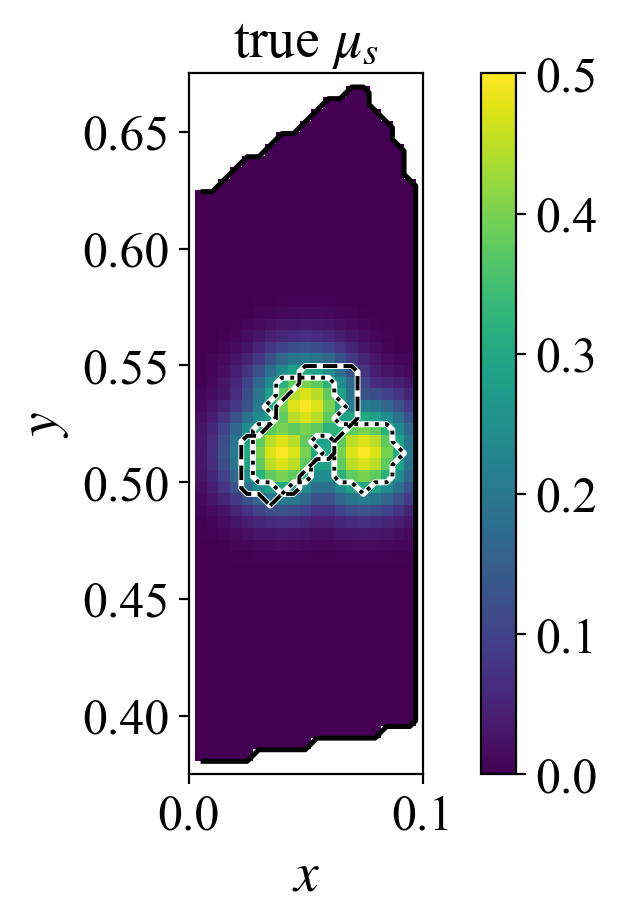}
    \includegraphics[scale=0.37]{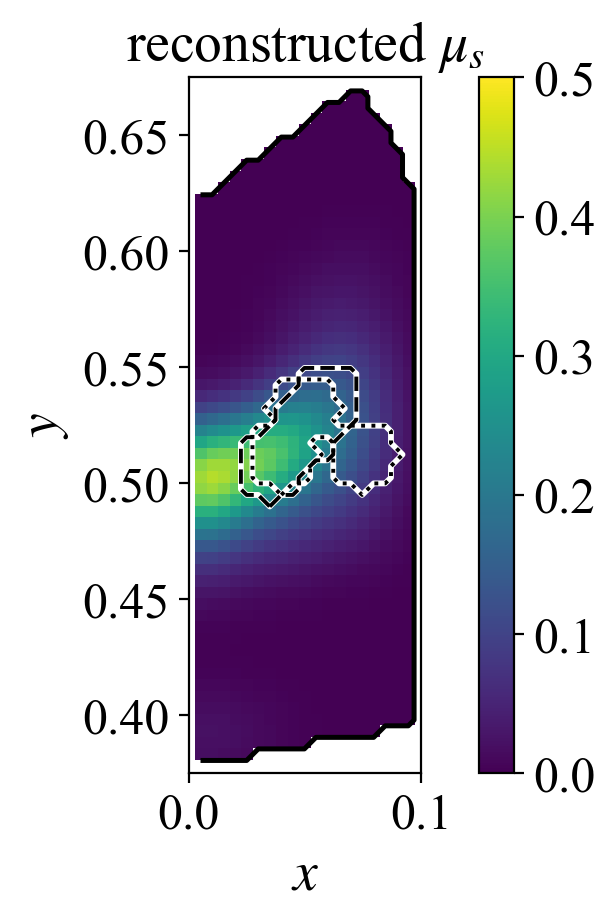}
    \includegraphics[scale=0.37]{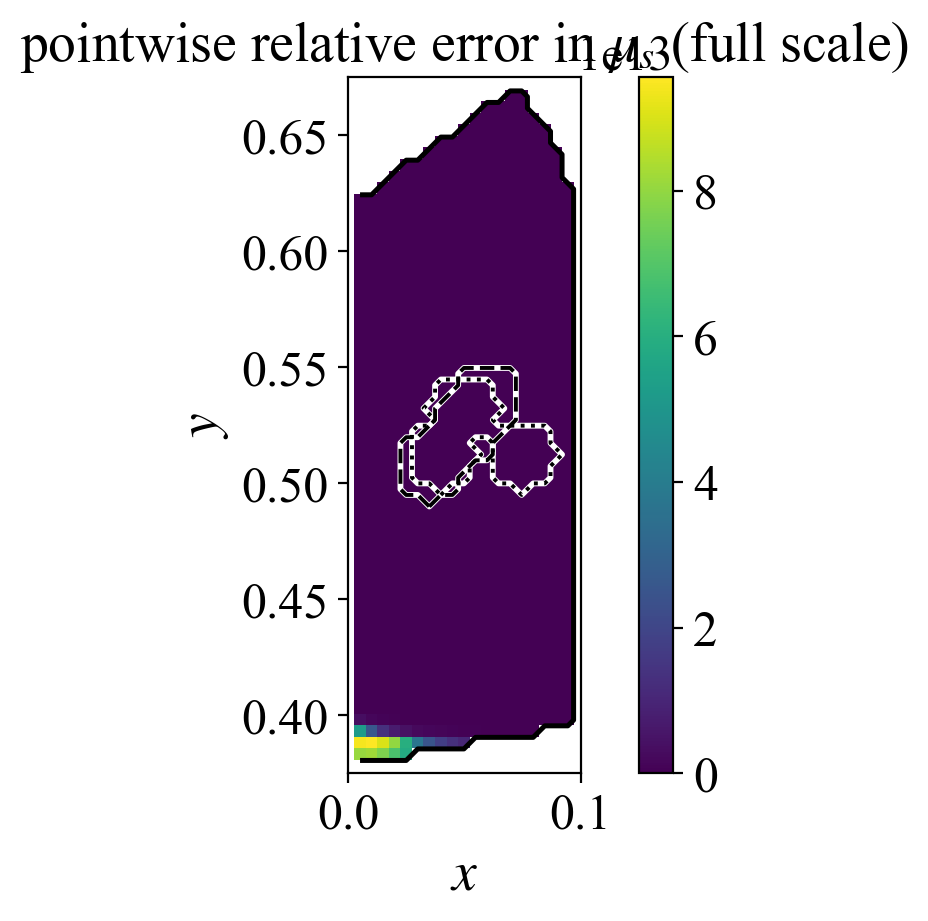}
\end{center}
\caption{Reconstruction of \(\mu_s\) for \(\eta=0.6\). From left to right: ground truth in the reconstruction window, reconstruction, and relative error.}
\label{fig:mus_reconstruction}
\end{figure}

\bibliographystyle{abbrv}
\bibliography{refs} 

\end{document}